\documentclass[11pt,reqno]{amsart}

\usepackage{amsmath,latexsym,amssymb,verbatim,amsbsy,amsthm,mathrsfs,tikz,times}
\usepackage[margin=1in]{geometry}

\title[Invariant measures for 3D stochastic Primitive Equations]{Existence and Regularity of Invariant Measures for the\\ Three Dimensional Stochastic Primitive Equations}

\date{\today}

\author{Nathan Glatt-Holtz}
\address{Department of Mathematics, Virginia Tech, Blacksburg, VA 24061}
\email{\tt negh@vt.edu}

\author{Igor Kukavica}
\address{Department of Mathematics, University of Southern California}
\email{\tt kukavica@usc.edu}

\author{Vlad Vicol}
\address{Department of Mathematics, Princeton University}
\email{\tt vvicol@math.princeton.edu}

\author{Mohammed Ziane}
\address{Department of Mathematics, University of Southern California}
\email{\tt ziane@usc.edu}

\usepackage{color}
\usepackage[colorlinks=true, pdfstartview=FitV, linkcolor=blue, citecolor=blue, urlcolor=blue]{hyperref}

\newcounter{quest}
\theoremstyle{plain}
\newtheorem{theorem}{Theorem}[section]

\newtheorem{lemma}[theorem]{Lemma}

\theoremstyle{definition}
\newtheorem{remark}[theorem]{Remark}

\def\tilde{\widetilde}
\numberwithin{equation}{section}

\newcommand{\indFn}[1]{1 \! \! 1_{#1}}
\newcommand{\sbasis}{(\Omega, \mathcal{F}, \mathbb{P}, \{\mathcal{F}_{t}\}_{t \geq 0}, W)}

\def\TT{{\mathbb T}}
\def\OO{{\mathcal M}}

\def\PP{{\mathcal P}}
\def\OO{{\mathcal O}}
\def\EE{{\mathbb E}}

\def \UU{\mathcal{U}}

\def\DD{\mathcal D}

\def\Leray{{\mathcal P}_H}

\def\ra{\rangle}
\def\la{\langle}

\def\Prb{{\mathbb P}}

\def\Pr{\mathop{\rm Pr}}

\def\Lip{\mathop{\rm Lip}}
\def\grad{\nabla_{h}}
\def\Lap{\Delta_{h}}
\def\dz{\partial_{z}}

\newcommand{\eps}{\varepsilon}
\renewcommand{\phi}{\varphi}

\def\sgn{\mathop{\rm sgn} \nolimits}

\begin{document}


\begin{abstract}
We establish the continuity of the Markovian semigroup associated with strong solutions of the stochastic 3D Primitive Equations, and prove 
the existence of an invariant measure. The proof is based on new moment bounds for strong solutions. 
The invariant measure is supported on strong solutions, but is furthermore shown to have higher regularity properties.  
\end{abstract}


\maketitle

\setcounter{tocdepth}{1}
\tableofcontents

\section{Introduction} \label{sec:intro}
The Primitive Equations are a fundamental model in climatology, forming the analytical basis for the most advanced large scale numerical
general circulation models in use today \cite{P,WP}.
In this work we develop the mathematical foundations for the study of statistically stationary states of a stochastic version of these equations.
The addition of white noise driven forcing terms to the equations of fluid dynamics has long played a significant role in the theory of turbulence 
\cite{Novikov1965, Eyink96} and in observation and data assimilation applications (cf.~e.g.~\cite{Bensoussan92} and more recently 
\cite{AzouaniOlsonTiti13, BessaihOlsonTiti13}). In particular, statistically invariant measures provide the natural framework to describe the long 
term behavior of oceanic and atmospheric processes.

The Primitive Equations are deduced from the full Boussinesq system by taking into account the Boussinesq approximation and the hydrostatic balance. 
In the Boussinesq approximation we assume that the density is constant in all but the buoyancy term and the equation of state, while in the hydrostatic 
balance approximation we use that the pressure gradient and the gravity forces dominate the rest of the terms in the vertical component of the momentum equation.

The mathematical theory of the Primitive Equations started with the works~\cite{LTW1, LTW2, LTW3}, where the mathematical framework of the equations 
was developed. In particular, the global existence of weak solutions, weak attractors, as well as the numerical analysis of the equations was obtained. 
On the other hand, the $H^2$ regularity of the linear problem was established in~\cite{Z1,Z2}. This led to the local existence of strong solutions of the 
Primitive Equations established in \cite{BGMR} and independently in \cite{HTZ} and \cite{TZ}. The global existence of strong solutions was proven 
in~\cite{CT} for the case of a cylindrical domain and Neumann boundary conditions on the top and the bottom of the ocean. The existence of solutions 
and the uniform bounds in the case of boundary conditions of Dirichlet type on the bottom of the ocean and with a variable bottom topography was 
established in~\cite{KZ1,KukavicaZiane07}.  The existence of the global attractor is given in the work~\cite{Ju} in the case of Neumann boundary 
conditions and follows from \cite{KZ3} for the general case of physical boundary  conditions. For additional background on the current state of the 
mathematical theory for the deterministic Primitive Equations, we refer the reader to~\cite{PetcuTemamZiane09} and references therein; see also~\cite{Brenier03,Ko,RousseauTemamTribbia08,Renardy09,KTVZ,CINT12,MW12,Wong12} and references therein for recent developments for the 
inviscid Primitive Equations.

The analysis of the stochastic Primitive Equations started on the two-dimensional version of the equations~\cite{GZ1} with the global existence
and uniqueness of pathwise (probabilistically strong) solutions for a general multiplicative noise. Similar results in the case of additive noise and 
$z$-weak solutions were obtained in~\cite{EwaldPetcuTemam2007}. The case of physical boundary conditions was addressed in~\cite{GT} and 
the companion work \cite{GlattHoltzTemam1}. The analysis of  three dimensional case started with the local existence and uniqueness  of pathwise 
strong solutions~\cite{DGT}. The global existence and uniqueness of strong pathwise solutions to the three-dimensional Primitive Equations was 
obtained recently in~\cite{DGHTZ12}. This required involved stopping time arguments to couple estimates in higher order $L^p$ spaces and 
anisotropic spaces depending only on vertical gradients, which allowed the authors to take advantage of the two dimensional nature of the pressure 
term. The case of an additive noise was treated in~\cite{GH07}.

In this paper we initiate the analysis of the Markovian framework and study invariant measures in this setting, 
an initial step towards understanding the time asymptotic behavior of solutions.  In Theorem~\ref{thm:existence}
we prove the existence of ergodic invariant measures, and in Theorem~\ref{thm:regularity} we study their higher regularity properties. 
Our work here will provide the initial foundations for the analysis of the unique ergodicity and mixing properties of the Primitive Equations 
which shall be given in our forthcoming work \cite{GHKVZ13b}.

Note that in~\cite{EvansGastler11} the existence of an invariant measure of the three-dimensional Primitive Equations with a {sufficiently time lagged} {kick-forcing} 
was established.  Recently in~\cite{Chueshov13} the squeezing property and exponential mixing have been obtained for this model.
This setting allows for an analysis which is much closer to the deterministic setting. The kick-forcing is particularly significant in the ``high-frequency'' 
limit where the number of stochastic kicks per unit time approaches infinity, cf.~e.g.~\cite{KuksinShirikyan03,KuksinShirikian12} and references therein.  This limit seems 
however to be out of the reach of the methods in \cite{EvansGastler11,Chueshov13}. In the present paper we work with a stochastic forcing that is given by this high frequency limit, i.e., 
white in time and colored in space. In fact, we may even allow the noise term to have a state-dependent structure.

For the analysis below we encountered a number of mathematical difficulties at the intersection of PDE theory and stochastic analysis. The basic foundations of the Makovian 
theory requires one to work in a functional setting where the pathwise existence and {continuous dependence on data} is evident. For instance, in the case of the 2D Navier-Stokes 
equations this space is $L^2$, which has the added advantage that the nonlinear term vanishes in $L^2$ estimates. The situation for the stochastic 3D Primitive Equations is different: 
while the uniqueness is classically measured in $L^2$,  in order to establish the continuous dependence on data, we need to work in strong spaces, such as $H^1$. In 
Theorem~\ref{thm:Feller} below we show that the Feller property holds in $H^1$. The problem which arises in strong spaces is the lack of a cancellation property for the 
nonlinear term. We overcome this difficulty and establish the Feller property by introducing a stopping time argument, combined with iterated weak moment bounds, and taking 
advantage of the parabolic smoothing effects inherent in the equation. 

We emphasize that {moment bounds} play an essential role in the existence, uniqueness, and regularity properties of statistically invariant states. As with the 
3D Navier-Stokes equations, in strong spaces there are no obvious cancellations in the nonlinear term, which precludes one from establishing suitable moment 
bounds in spaces that {compactly embed in $H^1$}. Indeed, the standard estimates give rise to cubic bounds for the nonlinearity, and thus the results in 
\cite{DGHTZ12} developed the analysis for the global existence of the stochastic Primitive Equations in a mostly {pathwise fashion}. Meanwhile, the recent 
paper~\cite{KukavicaVicol12} on the 2D Navier-Stokes equations in bounded domains, demonstrated how to establish {logarithmic type moment bounds in 
strong spaces}, where cancelations are unavailable. Using this insight we show in Theorem~\ref{thm:moments} below that {logarithmic} moments are available 
in $H^2$ (compactly embedded in $H^1$), which is sufficient to prove the {existence} of ergodic invariant measures supported on $H^1$. Moreover, these 
moment bounds a posteriori show that the invariant measures are in fact supported on $H^2$.

Since the estimates are technically involved, we present here only the analysis for the velocity part of the Primitive Equations.
The full Primitive Equations used in climate modeling account for physically fundamental rotational effects as well as thermodynamic 
and other important processes (salinity, moisture) responsible for density variations in the oceanic-atmospheric system. A more physical 
version of the equations will be addressed in~\cite{GHKVZ13b}.

\subsection{The stochastic Primitive Equations}
We consider the velocity part of the 3D stochastic Primitive Equations
\begin{align}
&d v + \left(-  \Delta v + v \cdot \grad v + w \dz v + \grad p \right) dt = \sigma(v) dW  \label{eq:P:1}\\
&\grad \cdot v + \dz w = 0\label{eq:P:2} \\
&\dz p = 0 \label{eq:P:3}
\end{align}
for the unknown velocity field $u=(v,w) = (v_1,v_2,w)$ and the pressure scalar $p$. Here the spatial variable $(x,z) = (x_1,x_2,z)$ belongs to 
$\OO :=\TT^2 \times (0,1)$. For simplicity of the presentation all the physical parameters (height, viscosity, size of periodic box) are set to $1$. 
We denote $\grad=(\partial_1,\partial_2)$, $\Lap = \partial_{11} + \partial_{22}$, $\nabla = (\partial_1,\partial_2,\dz)$, and $\Delta = \Lap + \partial_{zz}$. 
The stochastic terms in \eqref{eq:P:1} is understood as the formal expansion
\begin{align*}
\sigma(v) dW  = \sum_{k \geq 1} \sigma_{k}(v) dW^{k},
\end{align*}
where $W^{k}$ are a sequence of independent 1D Brownian motions
relative to some prescribed stochastic basis.  Natural compatibility,
boundedness, and Lipschitz conditions shall be imposed on $\sigma$.
We make these precise in \eqref{eq:sigma:cond:Lip} and \eqref{eq:sigma:cond:Bnd:L14}--\eqref{eq:sigma:cond:small} below.

We denote by $\Gamma_t = \TT^2 \times \{1\}$ the top and by $\Gamma_b = \TT^2 \times\{0\}$ the bottom boundary.
The results in this paper cover the case of Neumann boundary conditions (considered in \cite{CT}) associated to the system \eqref{eq:P:1}--\eqref{eq:P:3}, namely
\begin{align}
& w = 0 \quad \mbox{on} \quad \Gamma_b \cup \Gamma_t \label{eq:BC:1} \\
& \dz v = 0 \quad \mbox{on} \quad \Gamma_b \cup \Gamma_t \label{eq:BC:2} \\
& v(x,z) \mbox{ is } \TT^2\mbox{-periodic in } x \mbox{ for any } z\in(0,1). \label{eq:BC:3}
\end{align}
In view of the boundary condition for $w$ on $\Gamma_b$, the divergence-free nature of the velocity field implies
\begin{align}
w(x,z) = - \int_0^{z} \grad\cdot v(x,z') dz'. \label{eq:w}
\end{align}
The above identity for the diagnostic variable $w$ in terms of $v$ shall be used implicitly throughout.
Note that, the boundary condition for $w$ on $\Gamma_t$ implies  the  compatibility condition
\begin{align}
 \grad \cdot  Mv(x) = \grad \cdot \int_0^1v(x,z) dz = 0 \label{eq:compatibility},
\end{align}
where here and throughout the paper we denote by 
\begin{align}
Mf(x) = \int_{0}^{1} f(x,z)dz \label{eq:mean:def}
\end{align}
the vertical mean of a function $f(x,z)$.

We assume that the noise $\sigma(v)$ has zero mean on $\OO$, and as in~\cite{DGHTZ12},  that it obeys the condition
\begin{align}
\grad\cdot M \sigma(v) (x) = 0. \label{eq:sigma:cond}
\end{align}
Additional  boundedness and Lipschitz conditions are imposed on $\sigma$  in \eqref{eq:sigma:cond:Lip} below. Since under the imposed condition on 
$\sigma(v)$ the mean $\int_{\OO} v(x,z) dx dz$ is preserved by the evolution \eqref{eq:P:1}--\eqref{eq:P:3}, for simplicity of the presentation we consider 
initial data, and hence solutions, which have zero average over $\OO$. The proof given here may be adapted to the case of solutions whose average is 
not necessarily zero, by slightly adjusting the Gagliardo-Nirenberg inequalities to account for lower order terms.

Condition \eqref{eq:sigma:cond} on the noise implies that the pressure may be computed explicitly from
\begin{align}
-\Lap p = \grad\cdot \left( M (v \cdot \grad v) + M( w \dz v) \right)
\label{eq:p:equation}
\end{align}
and
\begin{align}
 \int_{\TT^2} p dx = 0. \label{eq:p:mean}
\end{align}
To prove \eqref{eq:p:equation}--\eqref{eq:p:mean}, we integrate \eqref{eq:P:1} in time, apply $\grad \cdot M$, and use the stochastic Fubini theorem (cf.~e.g.~\cite{ZabczykDaPrato1992}). Due to the periodic boundary condition \eqref{eq:BC:3} in $x$, it is direct to compute
\begin{align}
\grad p = {\mathcal R}_h M ( v\cdot \grad v + v \grad \cdot v) \label{eq:p}
\end{align}
where ${\mathcal R}_h$ is a composition of Riesz transforms that act in the two dimensional variable $x \in \TT^2$.

\begin{remark}[\bf Lateral boundary conditions]
As opposed to the periodic case considered here, in the presence of lateral boundary conditions one cannot appeal 
to the  explicit formula \eqref{eq:p}, and instead one has to use Sohr-von Wahl estimates to control the pressure. This 
in turn requires shifting the equations, which introduces a number of difficulties~\cite[Section 4]{DGHTZ12}. 
\end{remark}

\begin{remark}[\bf Vorticity formulation]
In the setting of this paper it will be convenient to use the evolution equation
\begin{align} 
d (\dz v) + \Bigl(v \cdot \grad (\dz v) + w \dz (\dz v) + (\dz v) \cdot \grad v - (\grad \cdot v) (\dz v)  - \Delta (\dz v)  \Bigr) dt =  \dz \sigma(v) dW  \label{eq:P:Vort}
\end{align}
obeyed by the vorticity $\dz v$, 
with the associated Dirichlet boundary conditions for $\dz v$ on $\Gamma_b \cup \Gamma_t$ (in view of \eqref{eq:BC:2}), and periodicity in the $x$-variable (in view of \eqref{eq:BC:3}). One may also refer to \eqref{eq:P:Vort} as the vorticity formulation of the 3D stochastic Primitive Equations.
\end{remark}

\subsection{Reformulation of the equations with periodic boundary conditions}

Inherent symmetries in the equations show that the solution of the Primitive Equations on $\TT^2 \times (0,1)$ with with boundary 
given by \eqref{eq:BC:1}--\eqref{eq:BC:3} may be recovered by solving the equations with periodic boundary conditions in both the $x$ and $z$ variables on the extended domain $\TT^2 \times (-1,1) =: \TT^3$, and restricting to $z\in(0,1)$. 

To see this, 
consider any solution of \eqref{eq:BC:1}--\eqref{eq:BC:3} with the Neumann boundary conditions on the top and bottom boundaries $\Gamma_b = \{z=0\}$, $\Gamma_{t} = \{z = 1\}$.  We perform an even extension of the solution across $\Gamma_b$
\begin{align*}
 v(x,z) = v(x,-z),  \quad \textrm{ for } (x,z) \in \TT^{2} \times (-1,0),
\end{align*}
which prescribes, with \eqref{eq:w}, an odd extension for $w(x, z) = -w(x,-z)$.  We also extend
$\sigma$ in an even fashion across $\Gamma_b$. To ensure the smoothness of the ensuing $\sigma$ across $\{z=0\}$, as in~\cite{DGHTZ12} we assume  
\begin{align}
\dz \sigma(v) = 0 \quad \mbox{on} \quad \Gamma_b \cup \Gamma_t.
\end{align}
In view of $\dz v = \dz \sigma(v)= 0$, on $\Gamma_b \cup \Gamma_t$, this extension keeps the solution sufficiently smooth in space (e.g. in $H^{2 + \eps}$ for $0 \leq \eps <1/2$).
Moreover, in view of the boundary condition $\dz v = 0$ on $\Gamma_t$, we have that $\dz v = 0$ on $\{z = -1\}$.
It is then not difficult to obtain that this extension of $v$ yields a solution of \eqref{eq:P:1}-- \eqref{eq:P:3} on
the domain $\TT^{2} \times (-1,1)$ but with periodic boundary conditions.

In view of the above remark, we henceforth consider the Primitive Equations on the extended domain $\TT^3 = \TT^2 \times (-1,1)$, with periodic boundary boundary conditions and with the compatibility condition
\begin{align*}
\grad \cdot M v = \grad \cdot \int_{-1}^1 v(x,z) dz = 0
\end{align*}
for all $x \in \TT^2$.  For initial data $v_{0}$ and noise $\sigma$ to the periodic equation that is even across $z = 0$, since 
the equation is invariant under the transformation $z \mapsto -z$ we may therefore recover any solution of \eqref{eq:P:1}-- \eqref{eq:P:3}
by taking this periodic solution and restricting it to $\TT^{2} \times (0,1)$.  As such we conclude that the periodic boundary conditions are in fact more general than the Neumann conditions used in previous works, cf.~e.g.~\cite{Petcu06,CT} and references therein.

\subsection{Well-posedness results for the stochastic 3D Primitive Equations}
Following the notation in~\cite{PetcuTemamZiane09,DGHTZ12} we consider the spaces
\begin{align*} 
H &= \left\{ v \in L^{2}(\TT^{3}) \colon \grad \cdot Mv = 0 \mbox{ in } \TT^{2}, v \mbox{ is } \TT^{3}\mbox{-periodic in $x$ and $z$}, \int_{\TT^{3}} v\, dx dz =0 \right\}
\end{align*}
and
\begin{align*} 
V &= H \cap H^{1}(\TT^3) 
\end{align*}
The Leray projection operator from $L^{2}(\TT^3)$ to $H$ is denoted by $\PP_{H}$, and the inner product on $H$ is denoted by $\la \cdot,\cdot \ra$. Lastly, in view of the periodic boundary conditions, the Stokes operator for the Primitive Equations
$
A = \PP_H (-\Delta)
$
is just the negative Laplacian when acting on functions in its domain
\begin{align*}
\DD(A) = H \cap H^2(\TT^3). 
\end{align*}

Let us now briefly recall some elements of infinite dimensional stochastic analysis needed in order to properly define the stochastic evolution equation \eqref{eq:P:1}--\eqref{eq:P:3}. We refer to~\cite{ZabczykDaPrato1992} for details regarding the general theory. Since we are working in the Markovian framework we consider only pathwise solutions of \eqref{eq:P:1}--\eqref{eq:P:3}, and as such we fix throughout this manuscript a stochastic basis $\mathcal{S} = \sbasis$, where $W$ is a cylindrical Brownian motion defined on an auxiliary separable Hilbert space $\UU$, adapted to the filtration $\{ {\mathcal F}_t\}_{t\geq 0}$.

Regarding assumptions on the state-dependent operator $\sigma$, we introduce the following commonly used notation. For two Banach spaces ${X}$, ${Y}$ we denote by $\Lip({X}, {Y})$ the set of {Lipschitz} continuous mappings, i.e., for $\Psi \in \Lip({X}, {Y})$, there exists $C>0$ such that 
\begin{align*}
\| \Psi(v_1) - \Psi(v_2) \|_{{Y}} \leq C \|v_1 - v_2 \|_{{X}}
\end{align*}
for any $v_1,v_2 \in {X}$. Note that in this case $\Psi \in \Lip({X},{Y})$ is thus also sublinear
\begin{align*}
\| \Psi(v)\|_{{Y}} \leq C (1 + \| v\|_{{X}})
\end{align*}
For a separable Hilbert space ${X}$ we denote by $L_2(\UU,{X})$ the space of Hilbert-Schmidt operators from $\UU$ to ${X}$. With the above introduced notation, we assume the noise term satisfies
\begin{align}
  \sigma \in \Lip(H;L_2(\UU,H))\cap \Lip(V;L_2(\UU,V)) \cap \Lip(D(A);L_2(\UU,D(A))),
  \label{eq:sigma:cond:Lip}
\end{align}
or, more concretely, that
\begin{align} 
\| \sigma(v_1) - \sigma(v_2) \|_{L_{2}(\UU,\DD(A^{j/2}))}^{2} \leq C \| v_1 - v_2\|_{\DD(A^{j/2})}^2 \label{eq:sigma:cond:Lip:quant}
\end{align}
holds for all $j \in \{0,1,2\}$.

The existence of  a unique strong pathwise (i.e., strong in both the probabilistic and the PDE sense) solution of \eqref{eq:P:1}--\eqref{eq:P:3} follows directly from~ \cite{DGT,DGHTZ12}.
More precisely, the following statement holds.
\begin{theorem}[\bf Existence and uniqueness of strong pathwise solutions]\label{thm:PathSol}
Fix a stochastic basis $\mathcal{S} = (\Omega$, $\mathcal{F}, \mathbb{P}, \{\mathcal{F}_{t}\}_{t \geq 0}, W)$ and an $\mathcal{F}_{0}$ measurable
random variable $v_{0} \colon \Omega \mapsto V$.  Moreover, assume that $\sigma$ obeys \eqref{eq:sigma:cond:Lip}. Then there exists a $V$-valued, predictable
process $v$ such that
\begin{align}
  v \in C([0,\infty),V) \cap L^{2}_{loc}((0,\infty), D(A)) \mbox{ a.s.,}
  \label{eq:ExReg}
\end{align}
and such that, for all $t \geq 0$
\begin{align}
  v(t) + \int_{0}^{t} \Bigl(- \Delta v + \Leray (v \cdot \grad v + w \dz v) \Bigr)ds = v_{0} + \int_{0}^{t} \sigma(v) dW
    \label{eq:ExEqn}
\end{align}
where $w$ is computed from $v$ via \eqref{eq:w}.
Furthermore, if 
$v$ and $\tilde{v}$ satisfy \eqref{eq:ExReg}--\eqref{eq:ExEqn}
and $v(0) = \tilde{v}(0)$ a.s., then
\begin{align*}
  \Prb( v(t) = \tilde{v}(t), \mbox{ for all } t \geq 0) = 1,
\end{align*}
that is, $v = v(t, v_{0})$ is pathwise unique.
\end{theorem}

In addition to \eqref{eq:sigma:cond:Lip}, we assume that the noise term $\sigma(v)$ obeys 
\begin{align}
\sum_k \|\sigma_k (v)\|_{L^{14}}^2 \leq C (1 + \|v\|_{L^{14}}^2)
\label{eq:sigma:cond:Bnd:L14}
\end{align}
and
\begin{align} 
\sum_k \|\dz \sigma_k(v)\|_{L^6}^2 \leq C (1 + \|\dz v\|_{L^6}^2)
\label{eq:sigma:cond:Bnd:W1z6}
\end{align}
for any $v \in \DD(A)$, where $C$ is a positive constant.
Lastly, we assume that the following condition  on the operator $\sigma$ holds:
 there exist constants $0 < \eps_\sigma  \ll 1$ and $C_\sigma > 0$, such that 
\begin{align} 
\| \sigma (v) \|_{L_2(\UU,L^2)}^2 \leq C_\sigma + \eps_\sigma \|v \|_{L^2}^2 
\label{eq:sigma:cond:small}
\end{align} 
for any $v \in H$. 
For instance, additive noise structures of the form 
\begin{align}
\sigma(v) dW = \sum_k \sigma_k(x) dW_k(t) \label{eq:additive:noise}
\end{align}
with $\{\sigma_k\}_{k \geq 1} \in L_2(\UU,L^2)$ of zero mean obeys condition \eqref{eq:sigma:cond:small}.

\begin{remark}
We emphasize that we do not consider condition \eqref{eq:sigma:cond:small} to be restrictive, but the contrary, they include the class of noise structures under which the existence of invariant measure is typically proven. Indeed the uniqueness of invariant measures is usually proven only with additive noise. 
It is more difficult to obtain uniqueness when the noise is multiplicative, and in principle this requires some sort of ``semi-definite condition'' on $\sigma$.  To the best of our knowledge, there are only a few scant works on multiplicative noise structures using ``coupling methods''; cf.~e.g.~\cite{Odasso06b,DO}.
\end{remark}

\subsection{The Markov semigroup and invariant measures}
We use the notation $v(t,v_0)$ to denote the unique pathwise strong solution of 
the Cauchy problem associated to \eqref{eq:P:1}--\eqref{eq:BC:3} with the initial data $v_0 \in L^2(\Omega;V)$. For a set $B \in {\mathcal B}(V)$, where ${\mathcal B}(V)$ denotes the family of Borel subsets of $V$, we define the  {transition functions}
\begin{align*}
	P_{t}(v_0, B) = \Prb( v(t, v_0) \in B)
\end{align*}
for any $t \geq 0$. Let $C_b(V)$ and $M_b(V)$ be the set of all real valued bounded
continuous , respectively bounded, Borel measurable functions on $V$.  For $t \geq 0$, define
the  ph{Markov semigroup}
\begin{align}
	P_t \phi(v_0) = \EE \phi( v(t, v_0)) = \int_{V} \phi(v) P_{t}(v_{0}, dv) \label{eq:Pt:def}
\end{align}
which maps $M_b(V)$ into itself. In the deterministic case, $v(t,v_0)$ depends continuously on the initial data $v_0$ in the topology of $V$. Correspondingly, in the  stochastic case we are able to prove the following statement.
\begin{theorem}[\bf Feller property]
\label{thm:Feller}
 The Markov semigroup $P_t$ associated to the 3D stochastic Primitive Equations is {Feller} on $V$, that is $P_t$ maps $C_b(V)$ into itself.
\end{theorem}
The proof of the Feller property is given in Section~\ref{sec:Feller} below. We emphasize here that this property does not follow directly from continuous dependence estimates and the Dominated Convergence theorem, as is standard for e.g. the 2D Navier-Stokes equations. Since for the 3D Primitive Equations we have to work in the phase space $V= H \cap H^1$, where continuous dependence may be proven, certain cancellations in the nonlinear term are absent. Indeed, the continuous dependence on data estimate of~Lemma~\ref{lem:cont:dependence:V} sees the norms of both solutions. In such ``strong norms'' this would be the case even for additive noise~\cite{CGHV13}. We overcome this difficulty by using a delicate stopping time argument and a careful iterated use of moment bounds for the equation. At this stage we also need to appeal to the inherent parabolic smoothing in the equations, via a Ladyzhenskaya-type estimate in the proof of Lemma~\ref{lem:moments:local:H2} below.  This argument of combining moment bounds with stopping time arguments and parabolic smoothing in order to establish the Feller property addresses a technical challenge present also in other SPDE, where there is a mismatch between spaces with cancelations are available, and spaces where the equations are well-posed (in the sense of Hadamard). As such the technique developed here is of independent interest.

We conclude this subsection with the definition of an invariant measure. Let $\Pr(V)$ be the set of Borealian probability measures on $V$.  An element
$\mu \in \Pr(V)$ is called an {invariant measure} for the Feller Markov semigroup associated to \eqref{eq:P:1}--\eqref{eq:BC:3} if
\begin{align}
  \int_{V} \phi(v_0) d \mu(v_0) = \int_{V} P_t  \phi(v_0) d \mu(v_0) \label{eq:invariant:def}
\end{align}
for every $t \geq 0$. In other words, $\mu$ is a fixed point for the dual semigroup $P_t^\ast$ for $t\geq 0$.

\subsection{Main results} \label{sec:intro:primitive}
The purpose of this paper is to prove the existence of an invariant measure for the three-dimensional stochastic Primitive Equations and to show that this measure is supported on functions that have regularity better than $H^2$.

\begin{theorem}[\bf Existence of an ergodic invariant measure] \label{thm:existence}
Assume that $\sigma$ obeys \eqref{eq:sigma:cond:Lip}, \eqref{eq:sigma:cond:Bnd:L14}, \eqref{eq:sigma:cond:Bnd:W1z6}, and  \eqref{eq:sigma:cond:small}.
Then, there exists an ergodic invariant measure $\mu \in \Pr(V)$ for the Feller Markov semigroup $P_t$ associated to the 3D stochastic Primitive Equations \eqref{eq:P:1}--\eqref{eq:BC:3}.
\end{theorem}
The next statement shows that the support of any such invariant measure lies on $H^2$-smooth functions.
\begin{theorem}[\bf Regularity of invariant measures] \label{thm:regularity}
Under the assumptions, of Theorem~\ref{thm:existence}, let $\mu \in \Pr(V)$ be any invariant measure for the 3D stochastic Primitive Equations. Then we have 
\begin{align*}
\int_{V} \log\Bigl(1 + \| \nabla (|v|^{7}) \|_{L^{2}}^{2} + \| \nabla (|\dz v|^3)\|_{L^2}^2  + \| \Delta v\|_{L^2}^2 \Bigr) d \mu (v) < \infty
\end{align*}
holds. In particular, any invariant measure $\mu$ is supported on $\DD(A) = H^2 \cap H$.
\end{theorem}

Usually, e.g.~for the stochastic 2D Navier-Stokes equations, proving the existence of invariant measures is direct:  From the energy inequality we obtain {moments} for both velocity in $L^2$ and vorticity in $L^2$, and since $H^1 \subset L^2$ is compact, the standard Kryloff-Bogoliubov procedure yields the existence of an invariant measure (see, e.g.~\cite{Debussche2011a,KuksinShirikian12}). The key here is that moments are available for the solution in spaces that compactly embed in a space where the equations depend continuously on the data (the Markov semigroup is Feller in $L^2$ for 2D Navier-Stokes).

The situation for the stochastic 3D Primitive Equations is much more complicated. As discussed in Theorem~\ref{thm:Feller} and the paragraph below it,  the Feller property is expected to hold only on $V$. The main difficulty which arises is that the moment bounds which lie at the heart of the existence theory, until this work were not available in spaces that compactly embed in $V$ (note that the moment bounds obtained in Section~\ref{sec:Feller} are for spaces that are larger than $V$). The issue here is that one has to estimate the equations in strong spaces, e.g.~in $V$, and hence the cancellation property for the nonlinear term (which was available in $H$) is not anymore available. We overcome the difficulty of establishing ``strong moments'' for the equation, using an idea recently used in~\cite{KukavicaVicol12} to obtain moment bounds for the stochastic 2D Navier-Stokes equations on a bounded domain. We show in Theorem~\ref{thm:moments} below that {logarithmic} moments are available in $H^2$. These logarithmic strong moments permit us to treat the cubic nonlinear term, and they are also sufficient to prove the {existence of ergodic invariant measures} supported in $V$, since $H^2 \subset V$ is compact. Moreover, the moments a posteriori show that the invariant measures are supported in fact on $H^2 \cap H$, and the estimate in Theorem~\ref{thm:regularity} holds.

\begin{remark}[\bf Uniqueness of the invariant measure]
In the context of the 2D Navier-Stokes equations, using quite delicate arguments one may prove that the invariant measure is unique, even under quite degenerate noise structures; cf.~\cite{FlandoliMaslowski1, ZabczykDaPrato1996,
Mattingly1,EMattinglySinai01,Mattingly2, BricmontKupiainenLefevere2001, KuksinShirikyan1,KuksinShirikyan2, MasmoudiYoung02,Mattingly03,MattinglyPardoux1,
HairerMattingly06, ChueshovKuksin08, Kupiainen10, HairerMattingly2011, Debussche2011a, KuksinShirikian12, FoldesGlattHoltzRichardsThomann2013} and references therein. 
The uniqueness of the invariant measure for the full 2D stochastic Primitive Equations (including temperature, gravity, and rotational effects) with more physically realistic boundary conditions will be  addressed in a forthcoming work~\cite{GHKVZ13b}.
\end{remark}

In the next statement we assert the higher regularity of the invariant measures.

\begin{theorem}[\bf Higher regularity for support of invariant measures] \label{thm:higher:regularity}
Let $\mu \in \Pr(V)$ be an invariant measure as in Theorem~\ref{thm:regularity}, and let $\eps \in (0,1/42]$. Then we have
\begin{align*}
\int_V \log \Bigl(1 + \|v\|_{H^{2+\eps}}^2 \Bigr) d\mu(v) < \infty
\end{align*}
and thus $\mu$ is supported on $H^{2+\eps} \cap H$.
\end{theorem}

Similarly to Theorems~\ref{thm:existence} and~\ref{thm:regularity}, the difficulty in establishing the above result lies in obtaining moment bounds for the solution in spaces that are smaller than $H^2$. This is achieved in Theorem~\ref{thm:higher:moments:eps} below. In view of Theorem~\ref{thm:higher:regularity} and the Gevrey-class regularity for the Primitive Equations established in~\cite{Petcu06}, we {conjecture that the invariant measures for the 3D primitive equation are in fact supported on $C^\infty$ functions}, with appropriate continuity and compatibility assumptions on the force, and the boundary conditions in this paper. For the 2D Navier-Stokes equations on a bounded domain, with Dirichlet boundary conditions, following an idea from~\cite{KukavicaVicol12}, we are indeed able to prove in~\cite{GHKVZ13a} that the (unique) invariant measure is supported on $C^{\infty}$ functions. For the Primitive Equations however, this question appears to be much more difficult (cf.~Proof of Theorem~\ref{thm:higher:regularity} below).

{\bf Organization of the paper.} In Section~\ref{sec:weak:moment} we establish ``weak'' moment bounds cf.~Lemmas~\ref{lem:energy:moment}, \ref{lem:Y:moment},  \ref{lem:L:moments:local}, and \ref{lem:moments:local:H2}, which allow us to bound the probability that certain stopping times are small cf.~Lemmas~\ref{lem:sigma:gamma}, \ref{lem:tau:kappa}, and \ref{lem:rho:lambda}. Using these preliminary estimates in Section~\ref{sec:Feller} we give the proof of the Feller property Theorem~\ref{thm:Feller}. In Section~\ref{sec:existence} we establish the strong moment bound Theorem~\ref{thm:moments}, which is the key ingredient in the proof of the existence of invariant measures Theorem~\ref{thm:existence}. In Section~\ref{sec:regularity} we establish the regularity of invariant measures and give the proof of Theorem~\ref{thm:regularity}. In Section~\ref{sec:higher:regularity} we establish an improved moment bound, Theorem~\ref{thm:higher:moments:eps}, which we then use to prove Theorem~\ref{thm:higher:regularity}.

\section{Weak moment bounds} \label{sec:weak:moment}
In order to show that the Markov semigroup $P_t$ is Feller on $V$ we first establish certain ``weak'' moment bounds for solutions of \eqref{eq:P:1}--\eqref{eq:BC:3}. In Lemma~\ref{lem:energy:moment} we consider the energy $\|v\|_{L^2}$, and in Lemma~\ref{lem:Y:moment}  the $L^6$ norm of $v$ coupled with the $L^2$ norm of the vorticity. We call these moments ``weak'' because they are in spaces which do not compactly embed in $V$, where we expect the Markov semigroup to be Feller. As such, these $L^2$ and $L^6$ moment bounds are not by themselves sufficient in order to obtain the existence of invariant measures via the Kryloff-Bogoliubov argument. 
In Lemma~\ref{lem:L:moments:local} we  obtain certain moment bounds for the $H^1$ norm of the solution, which are valid up to a stopping time. 

\begin{lemma}[\bf Moment bounds for the energy]
\label{lem:energy:moment}
Assume that $v_0 \in L^2(\Omega;L^2)$ and suppose that $\sigma$ obeys either \eqref{eq:sigma:cond:small} with $\eps_\sigma \leq \lambda_1^2$.
Then for any deterministic $T>0$ the  strong solution $v(t)$ of \eqref{eq:P:1}--\eqref{eq:BC:3} obeys
\begin{align} 
\EE \|v(T)\|_{L^2}^2 + \EE \int_0^T \|\nabla v(t)\|_{L^2}^2 dt \leq \EE \|v_0\|_{L^2}^2 + C T \label{eq:E:moment}
\end{align}
for a suitable positive constant $C$.
\end{lemma}

\begin{proof}[Proof of Lemma~\ref{lem:energy:moment}]
The It\=o lemma in $L^2$ applied to \eqref{eq:P:1}, combined with the cancellations
\begin{align*}
&\langle (v \cdot \grad + w \dz) v,v\rangle = \int (v \cdot \grad + w \dz) v \cdot v dx dz= - \frac 12 \int |v|^2 ( \grad \cdot v + \dz w) dx dz = 0
\end{align*}
and
\begin{align*}
&\langle \grad p, v\rangle = \int \grad p \cdot v dx dz = \int \grad p \cdot M(v) dx = - \int p \grad \cdot M(v) dx = 0 
\end{align*} 
yields
\begin{align} 
d \|v\|_{L^2}^2 
& = \left( - 2 \| \nabla v\|_{L^2}^2  + \sum_k \|\sigma_k (v)\|_{L^2}^2 \right) dt +2 \sum_k  \langle \sigma_k(v) ,v\rangle dW^k_t \label{eq:E:Ito}.
\end{align}
Integrating \eqref{eq:E:Ito} on $[0,T]$ and taking expected values of both sides, we obtain
\begin{align*} 
\EE \|v(T)\|_{L^2}^2 + \EE \int_0^T \| \nabla v \|_{L^2}^2 dt 
&= \EE \|v_0\|_{L^2}^2 + \EE \int_0^T \left( \|\sigma(v)\|_{L_2(L^2)}^2 - \| \nabla v\|_{L^2}^2 \right) dt\notag\\
&\leq \EE \|v_0\|_{L^2}^2 + \EE \int_0^T \left( \|  \sigma(v)\|_{L_2(L^2)}^2 - \lambda_1^2 \| v\|_{L^2}^2 \right) dt
\end{align*}
where $\lambda_1$ is the lowest eigenvalue of the (mean-zero) Laplacian on $\TT^3$, and we used  
\[
\EE  \int_0^T  \sum_k  \langle \sigma_k(v) ,v\rangle dW^k_t = 0
\]
for the martingale term.
Therefore, if the noise term satisfies the smallness condition \eqref{eq:sigma:cond:small}, with $\eps_\sigma \leq \lambda_1^2$ we have 
\[ 
\| \sigma(v)\|_{L_2(L^2)}^2 - \lambda_1^2 \| v\|_{L^2}^2 \leq C_\sigma.
\]
Therefore, there exists a constant $C$ that depends only on $C_\sigma, \lambda_1$, and $\alpha_\sigma$, such that 
\begin{align*} 
\EE \|v(T)\|_{L^2}^2 + \EE \int_0^T \|\nabla v(t)\|_{L^2}^2 dt \leq \EE \|v_0\|_{L^2}^2 + C T 
\end{align*}
holds for any deterministic time $T>0$, and the proof of the lemma is completed.
\end{proof}

Denote by $Y$ and $\bar Y$ the quantities
\begin{align}
Y = \| v \|_{L^6}^6 + \|\dz v\|_{L^2}^2, \quad \bar Y =  \| \nabla (|v|^3)\|_{L^2}^2 + \| \nabla \dz v \|_{L^2}^2 \label{eq:Y:def}
\end{align}
which are finite for $v \in V$, respectively $v \in \DD(A)$, in view of the three-dimensional Sobolev embedding. 

\begin{lemma}[\bf Moment bounds for $v$ in $L^6$ and $\dz v$ in $L^2$]
\label{lem:Y:moment}
For $v_0 \in L^2(\Omega; V)$ and any deterministic time $T>0$, the moment bound
\begin{align}
&\EE \sup_{t \in [0,T]}\log ( 1 + Y(t) ) 
+ \EE \int_0^T \frac{ \bar Y(t) }{1 +Y(t)} dt  \leq C \EE \log(1 + \|v_0\|_{H^1})  + C \EE \| v_0\|_{L^2}^2 + C T \label{eq:L6:H1z:moment}
\end{align}
holds,  for a suitable positive constant $C$, where $Y$ and $\bar Y$ are defined in \eqref{eq:Y:def} above.
\end{lemma}

\begin{proof}[Proof of Lemma~\ref{lem:Y:moment}]
In order to prove \eqref{eq:L6:H1z:moment}, we first obtain the equation obeyed  by $Y$. The It\=o lemma in $L^p$, with $p=6$, applied to \eqref{eq:P:1}, combined with the cancellation property $\langle (v \cdot \grad + w dz) v, v |v|^{4}\rangle = 0$, gives
\begin{align}
d \|v\|_{L^6}^6
&= \left( - \frac{10}{3} \| \nabla (v^3)\|_{L^2}^2 + 6 \langle \grad p, v |v|^{4} \rangle + 15 \sum_k \langle \sigma_k(v)^2, |v|^{4} \rangle \right) dt + 6 \sum_k \langle \sigma_k(v), v |v|^{4}\rangle dW^k_t\notag\\
&=:A_{L^6} dt + dM_{t,L^6}\label{eq:L6:Ito}.
\end{align}
In order to bound the pressure term, we use \eqref{eq:P:3} and \eqref{eq:p} to obtain
\begin{align*} 
\langle \grad p, v |v|^{4} \rangle 
&= \int \grad p \cdot v |v|^{4} dx dz  = \int \grad p \cdot M(v |v|^{4}) dx \notag\\
& \leq \|\grad p \|_{L^{3/2}_x}  \|M(v|v|^{4})\|_{L^{3}_x} \leq C \| M ( v\cdot \grad v + v \grad \cdot v)  \|_{L^{3/2}_x}  \|M(v|v|^{4})\|_{L^{3}_x}
\end{align*}
and thus, using the two-dimensional Gagliardo-Nirenberg inequality we arrive at
\begin{align}
6 \langle \grad p, v |v|^{4} \rangle 
&\leq C \left( \| v\|_{L^{6}_{x,z}} \|\nabla v\|_{L^2_{x,z}}  \right)  \left(  \| \grad M (v |v|^{4}) \|_{L^{6/5}_x} + \|M (v |v|^{4}) \|_{L^{6/5}_x} \right) \notag\\
&\leq C \| v\|_{L^{6}} \|\nabla v\|_{L^2}  \left( \| |v|^2 \|_{L^{3}_{x,z}} \| \nabla (|v|^3)\|_{L^2_{x,z}} +  \|v\|_{L^{6}_{x,z}}^{5} \right)  \notag\\
&\leq C \| v\|_{L^6} \|\nabla v\|_{L^2} \left( \|v\|_{L^6}^2 \| \nabla (|v|^3) \|_{L^2} + \|v\|_{L^6}^{5} \right)\notag\\
&\leq \frac{7}{3} \| \nabla (|v|^3) \|_{L^2}^2 + C \|v\|_{L^6}^{6} (1 + \|\nabla v\|_{L^2}^2). \label{eq:L6:1}
\end{align}
The term in \eqref{eq:L6:Ito} arising due to the It\=o correction is estimated using \eqref{eq:sigma:cond:Lip:quant} as
\begin{align} 
15 \sum_k \langle \sigma_k (v)^2, |v|^{4} \rangle 
&\leq 15 \sum_k \|\sigma_k(v)^2\|_{L^3} \| |v|^4 \|_{L^{3/2}} \notag\\
&\leq 15 \|v\|_{L^6}^4 \sum_k \|\nabla \sigma_k(v)\|_{L^2}^2
\leq C \|v\|_{L^{6}}^{4}  (1+\|\nabla v\|_{L^{2}}^{2} ) \label{eq:sigma:est:L6}
\end{align}
which combined with \eqref{eq:L6:1} yields
\begin{align} 
A_{L^6} \leq -\| \nabla (|v|^3) \|_{L^2}^2 + C (1+ \|v\|_{L^6}^{6})  (1 +  \|\nabla v\|_{L^2}^2) \label{eq:L6}
\end{align}
for some positive constant $C$.

The It\=o lemma in $L^2$, applied to the vorticity formulation \eqref{eq:P:Vort}, combined with the cancellation property $\langle (v \cdot \grad + w \dz) \dz v, \dz v \rangle = 0$, gives
\begin{align} 
d \| \dz v\|_{L^2}^2 
&= \left( - 2 \|\nabla \dz v\|_{L^2}^2 + 2 \langle - \dz v \cdot \grad v + (\grad \cdot v) \dz v, \dz v  \rangle   + \sum_k \|\dz \sigma_k(v)\|_{L^2}^2 \right) dt \notag\\
& \qquad \qquad + 2 \sum_k \langle \dz \sigma_k(v), \dz v\rangle dW^k_t\notag\\
&=: A_{H_z^1} dt + dM_{t,H_z^1} \label{eq:H1z:Ito}.
\end{align}
Upon integrating by parts in $x$ we have
\begin{align} 
\langle - \dz v \cdot \grad v &+ (\grad \cdot v) \dz v, \dz v \rangle
= \int (\grad \cdot v) |\dz v|^2 dx dz - \int (\dz v \cdot \grad v) \cdot \dz v dx dz  \notag\\
&\leq C \int |v| | \grad  \dz v|   |\dz v| dx dz 
\leq C \| v\|_{L^{6}} \| \nabla \dz v \|_{L^2} \| \dz v  \|_{L^{3}} \notag\\
& \leq C \| v\|_{L^{6}} \| \nabla \dz v \|_{L^2}  \left( \|v \|_{L^6}^{1/2} \| \nabla \dz v\|_{L^2}^{1/2} \right) \notag\\
&\leq \|\nabla \dz v\|_{L^2}^2 + C \| v\|_{L^6}^{6}, \label{eq:H1z:1}
\end{align}
where in the second inequality above we  appealed to the estimate $\|\dz v\|_{L^3} \leq C \|v\|_{L^6}^{1/2} \|\nabla \dz v \|_{L^2}^{1/2}$, which can be proven using integration by parts (see also Lemma~\ref{lem:vz:interpolation}). After using \eqref{eq:sigma:cond:Lip:quant} to bound
\begin{align} 
\sum_k \|\dz \sigma_k(v)\|_{L^2}^2  \leq C (1 + \| \nabla v\|_{L^2}^2)
\label{eq:sigma:est:H1z}
\end{align}
we obtain from \eqref{eq:H1z:1} that
\begin{align}
A_{H_z^1} 
&\leq - \|\nabla \dz v\|_{L^2}^2 + C (1 + \|\nabla v\|_{L^2}^2 + \|v\|_{L^6}^6)  \notag\\
&\leq - \|\nabla \dz v\|_{L^2}^2 + C (1+ \|v\|_{L^6}^{6})  (1 +  \|\nabla v\|_{L^2}^2) \label{eq:H1z}
\end{align}
for a suitable positive constant $C$.

To conclude the proof, we combine \eqref{eq:L6:Ito} and \eqref{eq:H1z:Ito} and obtain
\begin{align}
 d Y + \bar Y dt = \left(A_{L^6} + A_{H_z^1} + \bar Y\right) dt + (dM_{t,L^6} + dM_{t,H_z^1}),\label{eq:Y:Ito}
\end{align}
where $M_{t,\cdot}$ are the corresponding Martingale terms and $Y,\bar Y$ are as defined in \eqref{eq:Y:def}. Next, we appeal to an idea from \cite{KukavicaVicol12} and apply the It\=o formula to the $C^2$ function $\log (1+Y)$, to obtain
\begin{align}
d \log(1+Y) + \frac{\bar Y}{1+Y} dt 
&= \frac{1}{1+Y} \left(A_{L^6} + A_{H_z^1} + \bar Y\right) dt + \frac{1}{1+Y} (dM_{t,L^6} + dM_{t,H_z^1}) \notag\\
&\qquad - \frac{1}{2 (1+Y)^2} \sum_k \left( 6 \langle \sigma_k(v), v |v|^{4}\rangle + 2  \langle \dz \sigma_k(v), \dz v\rangle \right)^2 dt.
\label{eq:log:Y:Ito}
\end{align}
For a deterministic time $t$, we integrate \eqref{eq:log:Y:Ito} from $0$ to $t$, take a supremum over $t \in [0,T]$, apply the expected value, appeal to the bounds \eqref{eq:L6} and \eqref{eq:H1z}, and use the energy moment \eqref{eq:E:moment} to obtain
\begin{align}
& \EE \log(1+Y(T)) + \EE \int_0^T \frac{\bar Y(s)}{1+ Y(s)} ds \notag\\
& \leq \EE \log(1+ Y(0)) + C \EE \int_0^T (1+ \|\nabla v(s)\|_{L^2}^2) ds 
+ \EE  \sup_{t\in [0,T]} \left| \int_0^t \frac{1}{1+Y(s)} (dM_{s,L^6} + dM_{s,H_z^1}) \right|. \label{eq:Y:moment:2}
\end{align}
On the first term, we use the Sobolev embedding
\[
1 + Y(0) = 1 + \| v_0\|_{L^6}^6  + \|\dz v_0\|_{L^2}^2 \leq C (1 + \| v_0\|_{V})^6,
\] 
while on the second term, we appeal to \eqref{eq:E:moment}.
On the third term in \eqref{eq:Y:moment:2} we use the Burkholder-Davis-Gundy inequality, assumption \eqref{eq:sigma:cond:Lip:quant}, and the energy estimate \eqref{eq:E:moment}, in order to bound 
\begin{align} 
& \EE  \sup_{t\in [0,T]} \left| \int_0^t \frac{1}{1+Y(s)} (dM_{s,L^6} + dM_{s,H_z^1}) \right| \notag\\
& = \EE  \sup_{t\in [0,T]} \left| \sum_k \int_0^t \frac{1}{1+Y(s)} \left(6 \la \sigma_k(v),v|v|^4\ra + 2 \la \dz \sigma_k(v), \dz v\ra \right) dW^k_s \right| \notag\\
&\leq C \EE \left( \int_0^T \frac{1}{(1 + Y(t))^2} \left( (1+ \|\nabla v\|_{L^2}^2) \|v\|_{L^6}^{10} + (1+ \| \nabla v\|_{L^2}^2) \|\dz v\|_{L^2}^2 \right) dt \right)^{1/2} \notag\\
&\leq C \EE \left( \int_0^T 1 + \|\nabla v\|_{L^2}^2 dt \right)^{1/2} \leq C \EE \|v_0\|_{L^2}^2 + C (1+T) ,\label{eq:Y:moment:3}
\end{align}
Combining \eqref{eq:Y:moment:2}, \eqref{eq:Y:moment:3} and the energy moment bound \eqref{eq:E:moment}, we obtain
\begin{align}
\EE \sup_{t\in[0,t]} \log(1+Y(t)) + \EE \int_0^T \frac{\bar Y(s)}{1+ Y(s)} ds    \leq C \EE \log(1+ \|v_0\|_{V}) + C \EE \| v_0\|_{L^2}^2 + C (1+T)  \label{eq:Y:moment:1}
\end{align}
which concludes the proof of \eqref{eq:L6:H1z:moment}.
\end{proof}

For $\gamma>0$ and $v_0 \in V$, we introduce the stopping time 
\begin{align}
\sigma_\gamma(v_0) := \inf_{t\geq 0} \left\{ \sup_{s\in[0,t]}\| v(s,v_0)\|_{L^6}^4  + \int_0^t  \| \dz v(s,v_0)\|_{L^2}^2 \| \nabla \dz v(s,v_0)\|_{L^2}^2 ds > \gamma \right\}. 
\label{eq:sigma:gamma}
\end{align}
Note that we have $\sigma_\gamma(v_0) >0 $ when $\gamma \geq 1 + \|v_0\|_{L^6}^4$. The usefulness of the above stopping time shall become apparent in Section~\ref{sec:Feller}. We have introduced $\sigma_\gamma$ here since a direct consequence of Lemma~\ref{lem:Y:moment} is that we can control the probability that $\sigma_\gamma$ is smaller than a given deterministic time.
\begin{lemma}[\bf Estimate on $\sigma_\gamma$]
\label{lem:sigma:gamma}
 For $v_0 \in V$ and $\gamma \geq 2 + \|v_0\|_{L^6}^4$, let $\sigma_\gamma(v_0)$ be the stopping time defined in \eqref{eq:sigma:gamma}. Then, for any deterministic time $t>0$ we have
 \begin{align*} 
\Prb(\sigma_\gamma(v_0) \leq t ) \leq  \frac{C}{\log (\gamma/2)}  \Bigl( \EE \log(1+ \| v_0\|_{H^1}^2) + \EE \| v_0\|_{L^2}^2 + t \Bigr)
\end{align*}
for some positive constant $C$.
\end{lemma}
\begin{proof}[Proof of Lemma~\ref{lem:sigma:gamma}]
First, recall the quantities $Y = \|v\|_{L^6}^6 + \| \dz v \|_{L^2}^2$ and $\bar Y = \| \nabla (|v|^3)\|_{L^2}^2 + \| \nabla \dz v \|_{L^2}^2$ introduced in \eqref{eq:Y:def}.
By the definition of $\sigma_\gamma$, we have
\begin{align*}
\Prb( \sigma_\gamma(v_{0}) < t)
&\leq \Prb \left(  \int_{0}^{t} \|\dz v\|_{L^2}^{2} \|\nabla \dz v\|_{L^{2}}^{2} ds > \frac{\gamma}{2} \right) 
	+ \Prb \left( \sup_{[0,t]} \| v\|_{L^6}^4   > \frac{\gamma}{2} \right) \notag\\
&\leq \Prb \left(  \sup_{[0,t]}\left( 1 + \|v\|_{L^6}^6 + \| \dz v \|_{L^2}^2\right)^2 \int_{0}^{t} \frac{\|\nabla \dz v\|_{L^{2}}^{2}}{1 + \|v\|_{L^6}^6 + \| \dz v \|_{L^2}^2}ds > \frac{\gamma}{2} \right) \notag\\
& \qquad + \Prb \left(  \sup_{[0,t]} \left( 1 + \|v\|_{L^6}^6 + \| \dz v \|_{L^2}^2\right)^{2/3} > \frac{\gamma}{2}\right) \notag\\
&\leq  \Prb \left(  \sup_{ [0,t]} (1+Y) > \left(\frac{\gamma}{2}\right)^{1/4} \right)
+\Prb \left(   \int_{0}^{t} \frac{\bar Y}{1 + Y}ds >  \left(\frac{\gamma}{2}\right)^{1/2} \right)\notag\\
&\qquad  + \Prb \left(  \sup_{ [0,t]} (1+Y) > \left(\frac{\gamma}{2}\right)^{3/2} \right) \notag\\
&\leq 2 \Prb \left( \sup_{[0,t]} \log(1+ Y) > \frac 14 \log \frac{\gamma}{2} \right)  
+\Prb \left(   \int_{0}^{t} \frac{\bar Y}{1 + Y}ds > \left(\frac{\gamma}{2}\right)^{1/2} \right).
\end{align*}
In the last inequality we have also used that $\gamma \geq 2$.
Using the Chebyshev inequality and the moment bound obtained in Lemma~\ref{lem:Y:moment}, we conclude
\begin{align*}
\Prb( \sigma_\gamma(v_{0}) < t)
&\leq  \frac{8}{\log(\gamma/2)} \EE \left(  \sup_{s \in [0,t]}\log(1+Y(s)) \right)
+ \frac{{2}^{1/2}}{\gamma^{1/2}} \EE \left(   \int_{0}^{t} \frac{\bar Y(s)}{1 + Y(s)}ds \right) \notag\\
&\leq C  \left( \frac{1}{\log (\gamma/2)} + \frac{1}{\gamma^{1/2}} \right) \left( \EE \log(1+ \| v_0\|_{H^1}^2) + \EE \| v_0\|_{L^2}^2 + t \right)
\end{align*}
which gives the lemma upon noting that $\gamma\geq 2$ and hence $0 \leq \log (\gamma/2) \leq (\gamma/2)^{1/2}$.
\end{proof}
With $\sigma_\gamma$ defined in \eqref{eq:sigma:gamma}, we can also obtain a local in time moment bound for the $H^1$ norm of the solution, i.e., a bound at time $t \wedge \sigma_\gamma$. 
\begin{lemma}[\bf Local moment bounds for $v$ in $H^1$]
\label{lem:L:moments:local}
 Let $\gamma>0$, $v_0 \in V$, and $t>0$. We have
\begin{align}
  \EE &\sup_{s \in [0, t \wedge \sigma_\gamma]} \log(1 + \|\nabla v(s ,v_{0}) \|^{2}_{L^2}) 
  + \EE \int_{0}^{t\wedge \sigma_\gamma} \frac{\|\Delta v(s,v_{0})\|^{2}_{L^{2}}}{1 + \|\nabla v(s,v_{0})\|_{L^2}^{2}}ds 
  \notag\\
  &\leq C \EE \log(1 + \|v_{0}\|^{2}_{H^{1}}) + C(1+\gamma)(1+ t),
  \label{eq:FelMomBnds}
\end{align}
for a universal positive constant $C$.
\end{lemma}
We emphasize that we do not  obtain the moment bound \eqref{eq:FelMomBnds} up to any {deterministic} time $t>0$, only up to the stopping time $t \wedge \sigma_\gamma$, but this is sufficient for our purposes.
\begin{proof}[Proof of Lemma~\ref{lem:L:moments:local}]
As in the proof of~\cite[Theorem 3.2]{DGHTZ12}, apply $\nabla$ to \eqref{eq:P:1}, and use the $L^2$ It\=o lemma on the equation for $\nabla v$ to arrive at
\begin{align}
d \|\nabla v\|_{L^2}^2 
&= \left( - 2 \|\Delta v\|_{L^2}^2 + 2 \langle  v \cdot \grad v, \Delta v \rangle + 2 \langle  w \dz v, \Delta v\rangle  + \sum_k   \|\nabla \sigma_k(v)\|_{L^2}^2 \right) dt\notag\\
&\qquad + 2 \sum_k \langle \nabla \sigma_k(v), \nabla v\rangle dW^k_t \notag\\
&=: A_{H^1} dt + dM_{t,H^1}.
\label{eq:H1:Ito}
\end{align} 
Note that in view of the periodic boundary conditions, integrating by parts the term arising from the Laplacian in \eqref{eq:H1:Ito} yields
\begin{align*}
\langle - \nabla \Delta v, \nabla v \rangle 
&= \langle - \nabla \Lap v, \nabla v \rangle + \langle - \grad \partial_{zz} v, \grad v \rangle  + \langle - \dz \partial_{zz} v, \dz v \rangle \notag\\
&= \langle \nabla \grad v, \nabla \grad v \rangle + \langle \grad \dz v, \grad \dz v \rangle + \langle \partial_{zz}v , \partial_{zz} v \rangle = \|\Delta v\|_{L^2}^2.
\end{align*}
In \eqref{eq:H1:Ito} we have also used the cancellation property
\begin{align*}
\langle \grad p , \Delta v \rangle 
&= \langle \grad p, \Lap v \rangle + \langle \grad p , \partial_{zz} v \rangle  \notag \\ 
&=  - \langle p, \Lap \grad \cdot v \rangle - \langle \dz \grad p , \partial_{z} v \rangle  = \la p , \Lap \dz w \ra = - \langle \dz p, \Lap w \rangle = 0
\end{align*}
due to \eqref{eq:P:2} and \eqref{eq:P:3}.

In order to obtain a bound for $A_{H^1}$, we need to estimate the nonlinear terms 
$ I_1 =  \langle  v \cdot \grad v, \Delta v \rangle$ and  $\quad I_2 =  \langle  w \dz v, \Delta v \rangle$.
For the first term we have
\begin{align}
| I_1| 
&\leq \|v\|_{L^{6}} \|\grad v\|_{L^{3}} \|\Delta v\|_{L^2} \leq C \|v\|_{L^{6}} \left( \| \nabla v\|_{L^2}^{1/2} \|\Delta v\|_{L^2}^{1/2} \right)  \|\Delta v\|_{L^2} \notag\\ 
& \leq \frac{1}{2} \| \Delta v\|_{L^2}^2 + C \|\nabla v\|_{L^2}^2  \| v\|_{L^6}^4 \label{eq:AH1:1}.
\end{align}
The second nonlinear term is bounded using the anisotropic H\"older inequality and the Sobolev inequality as
\begin{align}
|I_2| 
&\leq \| w\|_{L^\infty_z L^4_x} \| \dz v \|_{L^2_z L^4_x} \|\Delta v\|_{L^2_{x,z}} 
\leq C \|\grad v\|_{L^2_z L^4_x} \|\dz v\|_{L^2_z L^4_{x}} \|\Delta v\|_{L^2_{x,z}}  \notag\\
&\leq C \left( \|\nabla v\|_{L^2}^{1/2} \|\Delta v\|_{L^2}^{1/2} + \|\nabla v\|_{L^2}\right) \left( \|\dz v\|_{L^2}^{1/2} \| \nabla \dz v\|_{L^2}^{1/2} \right) \|\Delta v\|_{L^2} \notag\\
&\leq \frac 12 \|\Delta v\|_{L^2}^2 + C \|\nabla v\|_{L^2}^2 (1 + \| \dz v\|_{L^2}^2 \|\nabla \dz v\|_{L^2}^2). \label{eq:AH1:2}
\end{align}
Also, by our assumption \eqref{eq:sigma:cond:Lip:quant} on $\sigma(v)$ we have 
\begin{align}
\sum_k   \|\nabla \sigma_k(v)\|_{L^2}^2 = \| \nabla \sigma(v)\|_{L_2(L^2)}^2 \leq C ( 1 + \|\nabla v\|_{L^2}^2 ) \label{eq:AH1:3},
\end{align}
which combined with \eqref{eq:AH1:1} and \eqref{eq:AH1:2} yields
\begin{align}
A_{H^1} \leq - \| \Delta v\|_{L^2}^2 + C \left(1 + \| \nabla v\|_{L^2}^2\right) \left( 1 + \| v\|_{L^6}^4 +  \| \dz v\|_{L^2}^2 \|\nabla \dz v\|_{L^2}^2 \right)  \label{eq:AH1}
\end{align}
for some positive constant $C$.

Now, we apply the It\=o lemma to the function $\log (1 + \|\nabla v\|_{L^2}^2)$ and obtain from \eqref{eq:H1:Ito}, 
combined with \eqref{eq:AH1:3} and \eqref{eq:AH1}, that
\begin{align}
&d \log (1 + \|\nabla v\|_{L^2}^2) + \frac{\|\Delta v\|_{L^2}^2}{1 + \|\nabla v\|_{L^2}^2} dt \notag\\
&\qquad \leq C \left( 1 + \| v\|_{L^6}^4 +  \| \dz v\|_{L^2}^2 \|\nabla \dz v\|_{L^2}^2 \right) dt + \frac{1}{1 + \|\nabla v\|_{L^2}^2}   dM_{t,H^1}  \notag\\
& \qquad \qquad + \frac{C}{(1+ \|\nabla v\|_{L^2}^2)^2} \sum_k  \langle \nabla \sigma_k(v),\nabla v\rangle^2 dt \notag\\
&\qquad \leq C \left( 1 + \| v\|_{L^6}^4 +  \| \dz v\|_{L^2}^2 \|\nabla \dz v\|_{L^2}^2 \right) dt 
+   2 \sum_k \frac{1}{1 + \|\nabla v\|_{L^2}^2}  \langle \nabla \sigma_k(v), \nabla v\rangle dW^k_t.
\label{eq:log:H1:1} 
\end{align}
To handle the martingale terms, apply the Burkholder-Davis-Gundy inequality and estimate
\begin{align*}
   \EE \sup_{s \in [0,t\wedge \sigma_\gamma]} &\left| \int_{0}^{s}  \sum_k \frac{\langle \nabla \sigma_k(v), \nabla v\rangle}{1 + \|\nabla v\|_{L^2}^2}   dW^k_{s'}  \right|
   \notag\\
   &\leq C \EE  \left( \int_{0}^{t\wedge \sigma_\gamma} 
   	\sum_k \left(\frac{\langle \nabla \sigma_k(v), \nabla v\rangle}{1 + \|\nabla v\|_{L^2}^2}  \right)^{2} ds \right)^{1/2}
   \leq C t^{1/2}.
\end{align*}
We now integrate the inequality \eqref{eq:log:H1:1} from $0$ to $s$, take a supremum over $[0, t\wedge \sigma_\gamma]$ and then expected values to obtain
\begin{align*}
&\EE \sup_{s \in [0,t\wedge \sigma_\gamma]} \left(\log(1+ \|\nabla v(s)\|_{L^2}^2) + \int_0^{t \wedge \sigma_\gamma} \frac{\|\Delta v(s)\|_{L^2}^2}{1 + \|\nabla v(s)\|_{L^2}^2} ds \right) \notag\\
&\qquad \leq \EE \log(1+ \|\nabla v_0\|_{L^2}^2)  
+ C \EE \int_0^{t\wedge \sigma_\gamma} \Bigl(1+ \|v (s)\|_{L^6}^4 + \| \dz v(s)\|_{L^2}^2 \|\nabla \dz v(s)\|_{L^2}^2\Bigr) ds  \notag\\
&\qquad \quad
+\EE \sup_{s \in [0,t\wedge \sigma_\gamma]} \left| \int_{0}^{s}  
   \sum_k \frac{\langle \nabla \sigma_k(v), \nabla v\rangle }{1 + \|\nabla v\|_{L^2}^2}  dW^k_{s'}  \right|
\notag\\
&\qquad \leq \EE \log(1+ \|\nabla v_0\|_{L^2}^2)  + C (1+t) (1+\gamma),
\end{align*}
which concludes the proof of \eqref{eq:FelMomBnds}. 
\end{proof}

For $\kappa > 0$ and $v_{0} \in V$ define the stopping time  
\begin{align}
   \tau_{\kappa}(v_{0}) := \inf_{t \geq 0} \left\{  \int_{0}^{t} \Bigl( \| \nabla v (s,v_0)\|_{L^2}^2 +  \|\Delta v (s,v_0)\|_{L^2}^2   + \|\nabla v(s,v_0)\|_{L^2}^{2} \|\Delta v(s,v_0)\|_{L^{2}}^{2} \Bigr)ds > \kappa \right\}
   \label{eq:tau:kappa}
\end{align}
where, as usual, $v(t,v_0)$ denotes as usual the strong solution of the 3D stochastic Primitive Equations with the initial condition $v_0$.  Similarly, for any $v_{0}^{(1)}, v_{0}^{(2)} \in V$, we introduce the notation
\begin{align}
  \tau_{\kappa}(v_{0}^{(1)},{v}_{0}^{(2)}) :=  \tau_{\kappa}(v_{0}^{(1)}) \wedge \tau_{\kappa}({v}_{0}^{(2)}) 
 \label{eq:tau:kappa:2}.
\end{align}
A  consequence of Lemma~\ref{lem:L:moments:local} is that we may estimate the probability of $\tau_\kappa$ being small. More precisely, we have.
\begin{lemma}[\bf Estimate on $\tau_\kappa$]
\label{lem:tau:kappa}
 Let $v_0\in V$, $\kappa \geq 2e$, and $t>0$.  For any $\gamma> 2 + \|v_0\|_{L^6}^4$, we have
 \begin{align}
\Prb( \tau_\kappa(v_0) < t ) 
&\leq  \frac{C}{\log \kappa}  \left( \EE \log(1 + \|v_{0}\|^{2}_{H^{1}}) + (1+\gamma)(1+ t) \right)  \notag\\
&\qquad + \frac{C}{\log \gamma}  \Bigl( \EE \log(1+ \| v_0\|_{H^1}^2) + \EE \| v_0\|_{L^2}^2 + t \Bigr)
\label{eq:Feller:3}
\end{align}
where $\tau_k$ is the stopping time defined in \eqref{eq:tau:kappa} above.
In particular, choosing $\kappa$ sufficiently large and $\gamma = \sqrt{\log \kappa}$
we have the estimate
\begin{align}
\Prb( \tau_\kappa(v_0) < t ) \leq \frac{C}{\log \log \kappa}  \Bigl(\EE \log(1+ \| v_0\|_{H^1}^2) + \EE \| v_0\|_{L^2}^2 + t + 1 \Bigr).
\label{eq:loglog:bnd:tau:kappa}
\end{align}
\end{lemma}

\begin{proof}[Proof of Lemma~\ref{lem:tau:kappa}]
 Letting $v = v(t,v_0)$ and $\gamma > 1+ \|v_0\|_{L^6}^4$, while recalling the stopping time $\sigma_\gamma(v_0)$ defined in \eqref{eq:sigma:gamma}, we obtain
\begin{align*}
  & \Prb( \tau_{\kappa}(v_{0}) < t) \notag\\
   &\leq \Prb (\tau_\kappa (v_0) < t , \sigma_\gamma(v_0) \geq t ) + \Prb( \sigma_\gamma(v_0) <  t) \notag\\
   &\leq \Prb \left(  \int_{0}^{t\wedge \sigma_\gamma} \Bigl( \| \nabla v\|_{L^2}^{2} + \|\Delta v\|_{L^{2}}^{2} + \|\nabla v\|_{L^2}^{2} \|\Delta v\|_{L^{2}}^{2} \Bigr) ds > \kappa \right) + \Prb( \sigma_\gamma(v_0) < t)
   \notag\\
   &\leq \Prb \left( C  \sup_{  [0,t \wedge \sigma_\gamma]}(1 + \|\nabla v\|_{L^2}^{2})^2 \int_{0}^{t\wedge \sigma_\gamma} \frac{\|\Delta v\|_{L^{2}}^{2}}{1 + \|\nabla v\|_{L^{2}}^{2}}ds > \kappa \right) +  \Prb( \sigma_\gamma(v_0) < t)
      \notag\\
   &\leq  \Prb \left(  \sup_{ [0,t\wedge \sigma_\gamma]}(1 + \|\nabla v\|_{L^2}^{2}) > {\kappa}^{1/4} \right)
             +\Prb \left(   \int_{0}^{t\wedge \sigma_\gamma} \frac{\|\Delta v\|_{L^2}^{2}}{1 + \|\nabla v\|_{L^2}^{2}}ds > \frac{ \sqrt{\kappa} }{C} \right) +  \Prb( \sigma_\gamma(v_0) < t)
             \notag\\
   & \leq  \Prb \left(  \sup_{  [0,t\wedge \sigma_\gamma]} \log (1 + \|\nabla v\|_{L^2}^{2}) > \frac 14 \log {\kappa}  \right)
             +\Prb \left(   \int_{0}^{t\wedge \sigma_\gamma} \frac{\|\Delta v\|_{L^{2}}^{2}}{1 + \|\nabla v\|_{L^2}^{2}}ds > \frac{ \sqrt{\kappa} }{C} \right) +  \Prb( \sigma_\gamma(v_0) < t),
\end{align*}
where $C\geq 1$ is  such that $2 \| \nabla v \|_{L^2} \leq C \|\Delta v\|_{L^2}$. 
Using the Chebyshev inequality, the local moment bound of Lemma~\ref{lem:L:moments:local}, and the estimate on $\sigma_\gamma$ given in Lemma~\ref{lem:sigma:gamma}, we arrive at
\begin{align*}
  & \Prb( \tau_{\kappa}(v_{0}) < t) \notag\\
  &\leq  \frac{4}{\log \kappa } \EE \left(  \sup_{s \in [0,t\wedge \sigma_\gamma]}\log(1+\| \nabla v\|_{L^2}^{2}) \right)
             + \frac{C}{\sqrt{\kappa}} \EE \left(   \int_{0}^{t \wedge \sigma_\gamma} \frac{\|\Delta v\|_{L^{2}}^{2}}{1 + \|\nabla v\|_{L^2}^{2}}ds \right) + \Prb( \sigma_\gamma(v_0) < t )  \notag\\
  &\leq      \frac{C}{\log \kappa}  \left( \EE \log(1 + \|v_{0}\|^{2}_{H^{1}}) + (1+\gamma)(1+ t) \right)  + \frac{C}{\log \gamma}  \Bigl( \EE \log(1+ \| v_0\|_{H^1}^2) + \EE \| v_0\|_{L^2}^2 + t \Bigr).
\end{align*}
This concludes the proof of the lemma.
\end{proof}

We introduce one more stopping time, designed to account for the instant parabolic regularization inherent in the equations. Namely, for $\lambda>0$, and $v_0 \in V$,  let
\begin{align}
\rho_\lambda(v_0)  := \inf_{t \geq 0} \left\{ t \| v(t, v_{0}) \|_{H^{2}}^{2} \geq \lambda \right\}.
\label{eq:rho:lambda}
\end{align}
and
\begin{align*}
\rho_\lambda(v_0,\tilde v_0) := \rho_\lambda(v_0) \wedge \rho_\lambda(\tilde v_0).
\end{align*}
In order to estimate the probability that $\rho_\lambda$ is small we need a suitable local moment bound on the $H^2$-norm of the solution, in the spirit of Lemma~\ref{lem:L:moments:local}.
\begin{lemma}[\bf Local moment bounds for $v$ in $H^2$]
 \label{lem:moments:local:H2}
 For $v_{0} \in V$ (deterministic) and $\kappa > 0$ we have
 \begin{align}
   \EE& \left(\sup_{s \in [0,  t \wedge\tau_{\kappa}(v_{0})] } s \| \Delta v(s, v_{0}) \|^{2} \right)
   \leq C \exp( C\kappa) (t + \kappa),
   \label{eq:H2:loc:moment}
\end{align}
for a universal positive constant $C$, where $\tau_{\kappa}(v_{0})$ is defined in
\eqref{eq:tau:kappa}.
\end{lemma}
\begin{proof}[Proof of Lemma~\ref{lem:moments:local:H2}]
Applying $-\Delta$ to \eqref{eq:P:1} and using the It\=o Lemma, we obtain
\begin{align}
d (t \| \Delta v\|_{L^2}^2) +2  t \| \nabla \Delta v\|_{L^2}^2 dt 
&= \| \Delta v\|_{L^2}^2 dt - 2 t  \la v \cdot \grad v + w \dz v, \Delta^2 v \ra dt + t \| \Delta \sigma(v)\|_{L^2}^2 dt \notag\\
& \qquad + 2t \la \Delta \sigma(v), \Delta v \ra dW_t.
\label{eq:H2:local:1}
\end{align}
The nonlinear term in \eqref{eq:H2:local:1} is bounded as 
\begin{align}
& 2t \left|  \la v \cdot \grad v + w \dz v, \Delta^2 v \ra \right|  \notag\\
&\qquad \leq 2 t \| \nabla \Delta v \|_{L^2} \left( \| \nabla( v \cdot \grad v) \|_{L^2} + \| \nabla (w \dz v) \|_{L^2} \right) = 2 t \| \nabla \Delta v \|_{L^2} (T_v + T_w) \label{eq:H2:local:2}
\end{align}
First, we estimate the $v$-term as
\begin{align}
T_v &\leq  \| \nabla v \cdot \grad v \|_{L^2} + \|  v \cdot \grad \nabla v \|_{L^2} \notag\\
&\leq   \| \nabla v \|_{L^4} \| \grad v \|_{L^4} + \| v\|_{L^\infty}  \| \grad \nabla v \|_{L^2} \leq C  \| \nabla v\|_{L^2}^{1/2} \| \Delta v\|_{L^2}^{3/2} \label{eq:H2:local:3}.
\end{align}
For the $w$-term we write 
\begin{align}
 T_w 
 &\leq \| \nabla w  \dz v \|_{L^2} + \| w \dz \nabla v \|_{L^2} \notag\\
 &\leq \| \nabla w \|_{L^\infty_z L^4_x} \| \dz v\|_{L^2_z L^4_x} + \| w \|_{L^\infty_z L^4_x} \|\dz \nabla v\|_{L^2_z L^4_x} \notag\\
 &\leq C \| \nabla v \|_{L^2}^{1/2} \| \Delta v \|_{L^2} \| \nabla \Delta v \|_{L^2}^{1/2}
 \label{eq:H2:local:4}
\end{align}
by appealing to the 2D Gagliardo-Nirenberg and the Poincar\'e inequalities.
Combining \eqref{eq:H2:local:1}--\eqref{eq:H2:local:4} and using the Poincar\'e  and the $\eps$-Young inequalities, we obtain
\begin{align}
d (t \| \Delta v\|_{L^2}^2) + t \| \nabla \Delta v\|_{L^2}^2 dt 
&\leq \| \Delta v\|_{L^2}^2 dt + C t \| \nabla v\|_{L^2}^2 \|\Delta v\|_{L^2}^4 dt + t \| \Delta \sigma(v)\|_{L^2}^2 dt \notag\\
&\qquad + 2t \la \Delta \sigma(v), \Delta v \ra dW_t.
\label{eq:H2:local:5}
\end{align}
The Burkholder-Davis-Gundy inequality implies that for any $0 \leq \tau_{a} \leq \tau_{b} \leq  t \wedge \tau_{\kappa}$
\begin{align}
&\EE \sup_{s \in [\tau_{a}, \tau_{b}]} \left|\int_{\tau_{a}}^{s} r \la \Delta \sigma(v), \Delta v \ra dW_{r} \right| \notag\\
&\leq C \EE \left|\int_{\tau_{a}}^{\tau_{b}} s^{2} |\la \Delta \sigma(v), \Delta v \ra|^{2} ds \right|^{1/2} 
\leq \frac{1}{2} \EE \sup_{s \in [\tau_{a},\tau_{b}] }(s \| \Delta v \|^{2}) + C \EE \int_{\tau_{a}}^{\tau_{b}}  s (1 +\| \Delta v \|^{2}) ds 
\label{eq:H2:local:6}
\end{align}
We integrate \eqref{eq:H2:local:5} from $\tau_{a}$ to $s$, take a supremum over $[\tau_{a}, \tau_{b}]$, apply expected values, and use 
\eqref{eq:H2:local:6} to obtain
\begin{align*}
 \EE& \left(\sup_{s \in [\tau_{a}, \tau_{b}] }(s \| \Delta v \|^{2})  \right) \notag\\
  &\leq 2\EE (\tau_{a} \| \Delta v(\tau_{a}) \|^{2})  + C \EE \int_{\tau_{a}}^{\tau_{b}}\left((1+ t)(1+\| \Delta v\|_{L^2}^2) + C t\| \nabla v\|_{L^2}^2\|\Delta v\|_{L^2}^{4}  \right)dt \notag\\
  &\leq 2\EE (\tau_{a} \| \Delta v(\tau_{a}) \|^{2})  + C \EE \int_{\tau_{a}}^{\tau_{b}}\left( (1 + | \nabla v\|_{L^2}^2\|\Delta v\|_{L^2}^{2})t\| \Delta v\|_{L^2}^2 + (1 + t) + \| \Delta v\|_{L^2}^2 \right)dt.
\end{align*}
and hence with a version of the stochastic Gronwall lemma, \cite{GZ1} and the definition of $\tau_{\kappa}$ we 
obtain \eqref{eq:H2:loc:moment} concluding the proof.
\end{proof}
Finally, using Lemma~\ref{lem:moments:local:H2}, we obtain estimates on the stopping time
$\rho_{\lambda}(v_{0})$.
\begin{lemma}[\bf Estimate on $\rho_{\lambda}$]
\label{lem:rho:lambda}
 Let $v_0\in V$, $\kappa \geq 2e, \lambda > 0$, and $t>0$.  For any $\gamma> 2 + \|v_0\|_{L^6}^4$, we have
 \begin{align}
\Prb( \rho_{\lambda}(v_{0}) < t)
&\leq \frac{C \exp( C\kappa) (t + 1)}{\lambda} + \frac{C}{\log \kappa}  \left( \EE \log(1 + \|v_{0}\|^{2}_{H^{1}}) + (1+\gamma)(1+ t) \right)  \notag\\
&\qquad + \frac{C}{\log \gamma}  \Bigl( \EE \log(1+ \| v_0\|_{H^1}^2) + \EE \| v_0\|_{L^2}^2 + t \Bigr),
\label{eq:rho:lambda:bnd}
 \end{align}
for a sufficiently large universal constant $C$. In particular for $\lambda >0$ sufficiently large, with  $\kappa = \log (\sqrt{\lambda})/C$, where $C$ is the constant appearing in \eqref{eq:rho:lambda:bnd}, and with $\gamma= \sqrt{\log \kappa}$ we have
 \begin{align}
   \Prb( \rho_{\lambda}(v_{0}) < t) \leq \frac{C}{\log \log \log( \lambda)}(\EE \| v_0\|_{H^1}^2 + t +1).
   \label{eq:rho:lambda:bound}
 \end{align}
\end{lemma}
\begin{proof}[Proof of Lemma~\ref{lem:rho:lambda}]
Using Lemma~\ref{lem:moments:local:H2} and Lemma~\ref{lem:tau:kappa}, we have
\begin{align*}
\Prb ( \rho_{\lambda}(v_{0}) < t)  
&\leq \Prb \left( \sup_{s \in [0,  t] } s \| \Delta v(s) \|^{2}  \geq \lambda \right) \notag\\
&\leq \frac{1}{\lambda}\EE \left( \sup_{s \in [0,  t \wedge \tau_{\kappa}] } s \| \Delta v(s) \|^{2}  \right) + \Prb( \tau_{\kappa} < t) \notag\\
&\leq	\frac{C \exp( C\kappa) (t + \kappa)}{\lambda} + \frac{C}{\log \kappa}  \left( \EE \log(1 + \|v_{0}\|^{2}_{H^{1}}) + (1+\gamma)(1+ t) \right)  \notag\\
&\qquad+ \frac{C}{\log \gamma}  \Bigl( \EE \log(1+ \| v_0\|_{H^1}^2) + \EE \| v_0\|_{L^2}^2 + t \Bigr)
\end{align*}
and the desired estimate is obtained.
\end{proof}

\section{Feller property} \label{sec:Feller}
Equipped with the moment bounds and the estimates on the stopping time $\sigma_\gamma$ established in Section~\ref{sec:weak:moment}, we now have the necessary tools to establish the Feller property for the Markov semigroup $P_t$, i.e., to give the proof of Theorem~\ref{thm:Feller}. The continuous dependence on the initial data in the topology of $V$ is given quantitatively in the next lemma.

\begin{lemma}[\bf The difference of two solutions in the topology of $V$]
 \label{lem:cont:dependence:V}
 Let $v_0^{(i)} \in V$, where $i\in \{1,2\}$, and denote by $v^{(i)}(t)$ the corresponding strong solution of the 3D stochastic Primitive Equations \eqref{eq:P:1}--\eqref{eq:BC:3}. Then, for any $\kappa, t>0$, we have
\begin{align}
  \EE \sup_{s \in [0, t \wedge \tau_{\kappa}(v_0^{(1)},v_0^{(2)})]} \| \nabla v^{(1)}(s) - \nabla v^{(2)}(s)\|^{2}_{L^2}
  \leq C e^{C(\kappa+t)}   \EE \|\nabla v^{(1)}_0 - \nabla v^{(2)}_0\|^{2}_{L^2} 
   \label{eq:cont:dependence:V}
\end{align}
for some positive constant $C$, where the stopping time $\tau_\kappa(\cdot,\cdot)$ is  defined in \eqref{eq:tau:kappa:2} above.
\end{lemma}

\begin{remark}[\bf Feller property for additive noise]
Typically, e.g.~for the 2D Navier-Stokes equations with additive noise, the Feller property follows directly from the dominated convergence theorem. In this much simpler situation, the noise terms exactly cancel when one caries out continuous dependence estimates, and one obtains a bound like \eqref{eq:cont:dependence:V}
pathwise.  Note moreover that in this case, since the phase space is usually $L^{2}$, the exponent $\kappa$ in \eqref{eq:cont:dependence:V} only involves one of the two solutions due to cancelation. 

For the 3D Primitive Equations if we were to restrict ourselves to the case of additive noise, a suitably modified version of the above lemma can also be obtained pathwise. Indeed, for additive noise the right side of \eqref{eq:P:diff} below equals $0$, and hence there is no need to take expected values, and we obtain \eqref{eq:cont:dependence:V} almost surely. Nevertheless, the bound we obtain still depends on both the norms of $v_0^{(1)}$ and of $v_0^{(2)}$, so an involved estimate like \eqref{eq:split:Lip:approx1} below is still needed. The are other approaches to overcome this difficulty, but the advantage of the analysis we give below is that it is suitable for treating multiplicative noise.
\end{remark}

The proof of the the above lemma is based on the stochastic Gr\"onwall lemma and energy estimates. For clarity of the presentation we present the details of the proof at the end of this section. We now give the proof of the Feller property in $V$, assuming Lemma~\ref{lem:cont:dependence:V} holds. 
\begin{proof}[Proof of Theorem~\ref{thm:Feller}]
Fix $t>0$ and $\phi \in C_b(V)$ and any $v_{0} \in V$. Note that the norms $\|v\|_{H^1}$ and $\|\nabla v\|_{L^2}$ yield equivalent topologies on $V$ in view of the Poincar\'e inequality. Given an arbitrary $\eps>0$, we need to find $\delta  \in (0,1)$, such that 
\begin{align*} 
  |P_{t} \phi(v_{0}) - P_{t} \phi(\tilde{v}_{0})| = | \EE ( \phi(v(t,v_0)) - \phi(v(t,\tilde v_0)) ) |  \leq \eps
\end{align*}
holds for any  $\tilde{v}_{0} \in B_{V}(\delta,v_{0})$, where $B_{V}(\delta,v_{0})$ denotes the ball of radius $\delta$ in $V$
around $v_{0}$.

In order to use \eqref{eq:cont:dependence:V} we approximate $\phi$ in a suitable
way with a Lipschitz continuous function on $V$, and then employ a delicate stopping time argument
involving $\sigma_{\gamma}$, $\tau_{\kappa}$, and $\rho_{\lambda}$ introduced above in 
\eqref{eq:sigma:gamma}, \eqref{eq:tau:kappa}, and \eqref{eq:rho:lambda}.  We introduce the space of Lipschitz continuous functions on $V$
\[
\Lip(V) = \{\tilde \phi \in C_b(V) \colon \tilde \phi \mbox{ is Lipschitz continuous on } V \}.
\]
Observe now that for any $\tilde{v}_{0} \in B_{V}(1,v_0)$ and $\kappa, \lambda > 0$ we have
\begin{align}
&|\EE (\phi(v(t, v_{0})) - \phi(v(t,\tilde{v}_{0})))| \notag\\
&\qquad \leq |\EE (\phi(v(t, v_{0})) - \phi(v(t,\tilde{v}_{0}))) \indFn{\rho_{\lambda}(v_{0},\tilde{v}_{0}) > t} | + |\EE (\phi(v(t, v_{0})) - \phi(v(t,\tilde{v}_{0}))) \indFn{\rho_{\lambda}(v_{0},\tilde{v}_{0}) \leq t}| \notag\\
&\qquad \leq  |\EE (\phi(v(t, v_{0})) - \phi(v(t,\tilde{v}_{0}))) \indFn{\rho_{\lambda}(v_{0},\tilde{v}_{0}) > t} | + 2 \|\phi\|_{\infty} \Prb(\rho_{\lambda}(v_{0},\tilde{v}_{0}) \leq t) \notag\\         
&\qquad \leq |\EE (\phi(v(t, v_{0})) - \phi(v(t,\tilde{v}_{0}))) \indFn{\rho_{\lambda}(v_{0},\tilde{v}_{0}) > t}  \indFn{\tau_{\kappa}(v_{0},\tilde{v}_{0}) \geq t}| \notag\\
&\qquad \qquad + 2\|\phi\|_{\infty} \Prb(\tau_{\kappa}(v_{0},\tilde{v}_{0}) < t) +   2 \|\phi\|_{\infty} \Prb(\rho_{\lambda}(v_{0},\tilde{v}_{0}) \leq t) \notag\\
&\qquad = T_{1} + T_{2} +T_{3},\label{eq:split:Lip:approx1}
\end{align}
where the stopping time $\rho_{\lambda}(v_{0},\tilde{v}_{0})$ is defined in \eqref{eq:rho:lambda}, and $\tau_{\kappa}(v_0,\tilde v_0)$ is defined in \eqref{eq:tau:kappa}.

To address $T_{1}$ in \eqref{eq:split:Lip:approx1} we approximate the given $\phi \in C_{b}(V)$ by a 
an element $\tilde \phi \in \Lip(V)$  to be chosen below. Note that on the set $\{ \rho_\lambda(v_0,\tilde v_0) > t\}$ we have $v(t,v_0), v(t,\tilde v_0) \in B_{H^2}(\kappa/t)$. Hence, for any 
$ \lambda, \kappa > 0$ we estimate
\begin{align}
T_{1} 
&\leq 2 \sup_{v \in B_{H^{2}}(\kappa/t)}  | \phi(v) - \tilde \phi (v)| + |\EE (\tilde \phi(v(t, v_{0})) - \tilde \phi(v(t,\tilde{v}_{0}))) \indFn{\tau_{\kappa}(v_{0},\tilde{v}_{0}) > t} |  \notag\\
&\leq 2 \sup_{v \in B_{H^{2}}(\kappa/t)}  | \phi(v) - \tilde \phi (v)|  +  \| \nabla \tilde\phi \|_{\infty} \EE \|  \nabla v(t \wedge \tau_{\kappa}(v_0,\tilde{v}_{0}),v_{0} ) - \nabla v(t\wedge \tau_{\kappa}(v_0,\tilde{v}_{0}),\tilde{v}_{0}) \|_{L^2} 
\label{eq:split:Lip:approx2}
\end{align}
where $\| \nabla \tilde \phi  \|_{\infty}$ is the Lipschitz constant of $\tilde \phi $.  
By Lemma~\ref{lem:cont:dependence:V} and Jensen's inequality we have  
\begin{align} 
\EE \|  \nabla v(t \wedge \tau_{\kappa}(v_0,\tilde{v}_{0}),v_{0} ) - \nabla v(t\wedge \tau_{\kappa}(v_0,\tilde{v}_{0}),\tilde{v}_{0}) \|_{L^2} 
\leq  C e^{C (\kappa + t)} \| \nabla v_0 - \nabla \tilde v_0\|_{L^2}. \label{eq:Feller:1}
\end{align}
Combining \eqref{eq:split:Lip:approx2} with \eqref{eq:Feller:1} we finally conclude
\begin{align*}
T_{1} \leq 2 \sup_{v \in B_{H^{2}}(\kappa/t)}  | \phi(v) - \tilde \phi (v)|   + \| \nabla \tilde \phi \|_{\infty} C e^{C (\kappa + t)} \|  v_0 - \tilde v_0\|_{H^{1}}.
\end{align*}
We now bound $T_{2}$.  According to  \eqref{eq:loglog:bnd:tau:kappa}, we have for any $\kappa >0$ sufficiently large 
\begin{align} 
T_2 &\leq 2 \| \phi\|_{\infty} \Prb(\tau_\kappa(v_0,\tilde v_0) < t)  \notag\\
&\leq 2 \| \phi\|_{\infty}  \Bigl( \Prb (\tau_\kappa(v_0) < t ) + \Prb(\tau_\kappa(\tilde v_0) < t) \Bigr)
\leq \frac{C \| \phi\|_{\infty} }{\log \log \kappa}  \Bigl(\| v_{0}\|_{H^{1}}^{2} + t + 1 \Bigr).\label{eq:Feller:2}
\end{align}
Finally, we estimate $T_3$. Here with \eqref{eq:rho:lambda:bound} we obtain
\begin{align*}
  T_{3} \leq  2 \|\phi\|_{\infty} \bigl(\Prb(\rho_{\lambda}(\tilde{v}_{0}) \leq t) +\Prb(\rho_{\lambda}(v_{0}) \leq t) \bigr)
  \leq \frac{C \|\phi\|_{\infty}}{\log \log \log \lambda}  \Bigl(\| v_{0}\|_{H^{1}}^{2} +   t + 1 \Bigr)
\end{align*}
It remains to choose our parameters in an appropriate order. First,  let $\kappa = (\log \lambda)/C$, for a sufficiently large universal constant $C$, so that
\begin{align} 
T_2 + T_3 \leq \frac{C \|\phi\|_{\infty}}{\log \log \log \lambda}  \Bigl(\| v_{0}\|_{H^{1}}^{2}   + t + 1 \Bigr)
\label{eq:Feller:17}.
\end{align}
Next, we choose $\lambda = \lambda(t, \|\phi\|_{\infty}, \|v_0\|_{H^1})$ sufficiently large, so that  
\[ 
T_2 + T_3 \leq \frac{\eps}{2}
\] 
This choice of $\lambda$ automatically fixes the value of $\kappa$, and hence also the radius of the ball $B_{H^2}(\kappa/t)$. Therefore, since this ball is a relatively compact subset of $V$, and $\phi \in C_b(V)$, we may choose $\tilde \phi \in \Lip(V)$ so that 
\[
2 \sup_{v \in B_{H^{2}}(\kappa/t)}  | \phi(v) - \tilde \phi (v)| \leq \frac{\eps}{4}
\]
This choice of $\tilde \phi$ also fixes $\|\nabla \tilde \phi\|_{\infty}$ and hence we may at last choose $\delta$ sufficiently small, so that
\[
\| \nabla \tilde \phi \|_{\infty} C e^{C (\kappa + t)} \|  v_0 - \tilde v_0\|_{H^{1}} \leq \frac{\eps}{5}
\]
holds for any $\tilde v_0 \in B_{V}(\delta,v_0)$. This concludes the proof of the Feller property in $V$. 
\end{proof}

We now turn to the proof of Lemma~\ref{lem:cont:dependence:V}.
\begin{proof}[Proof of Lemma~\ref{lem:cont:dependence:V}]
For $v_0^{(j)} \in V$, let  $(v^{(j)},w^{(j)},q^{(j)})$, for $j=1,2$, be the two corresponding strong solutions of the stochastic 3D Primitive Equations,  and 
define
\begin{align*} 
\max_{j=1,2} \| v^{(j)}(t)\|_{H^1} &=: M(t) \quad \mbox{and} \quad
\max_{j=1,2} \| v^{(j)}(t)\|_{H^2} =: \bar M(t),
\end{align*}
for all $t\in [0,T]$. Note that $M$ and $\bar M$ are random functions of time. Denote the difference of solutions as
\begin{align*} 
v = v^{(1)}-v^{(2)}, \quad w = w^{(1)}-w^{(2)}, \quad q = q^{(1)}-q^{(2)}.
\end{align*}
We also denote by
\begin{align*}
 R = \| \nabla v\|_{L^2} \quad \mbox{and} \quad \bar R = \| \Delta v\|_{L^2}
\end{align*}
the $H^1$ respectively $H^2$ norm of the difference (remainder).
The equation for the difference is
\begin{align} 
dv + \left(- \Delta v + v^{(1)} \cdot \grad v + w^{(1)} \dz v + v \cdot \grad v^{(2)} + w \dz v^{(2)} + \grad q \right) dt = \left( \sigma(v^{(1)}) - \sigma(v^{(2)}) \right) dW.
\label{eq:P:diff}
\end{align}
After applying $\nabla$ to \eqref{eq:P:diff} and then using the It\=o lemma in $L^2$, we obtain
\begin{align} 
\frac 12 d R^2 + {\bar R}^2 dt
&= -  \la \partial_i v^{(1)} \cdot \grad v, \partial_i v \ra dt
-  \la \partial_i w^{(1)} \dz v,  \partial_i v \ra dt
- \la \partial_i v \cdot \grad v^{(2)} , \partial_i v \ra  dt \notag \\
&\quad \quad 
-  \la \partial_i w \dz v^{(2)} , \partial_i v \ra dt
- \la v \cdot \grad \partial_i v^{(2)} , \partial_i v \ra dt
- \la w \dz \partial_i v^{(2)} , \partial_i v \ra  dt \notag \\
& \quad \quad + \sum_k \la \partial_i \sigma_k(v^{(1)}) - \partial_i \sigma_k(v^{(2)}), \partial_i v \ra dW^k \notag\\
&\quad \quad + \frac 12 \sum_k  \left( \la \partial_i \sigma_k(v^{(1)}) - \partial_i \sigma_k(v^{(2)}), \partial_i v \ra \right)^2 dt \notag\\
&=: - (T_1 + T_2 + T_3  + T_4 + T_5 + T_6) dt + T_7 dW + \frac 12 T_8 dt, \label{eq:diff:H1:1}
\end{align}
where the summation over the repeated index $i$ is over $i=1,2,3$. 
Note that the pressure term vanishes upon integrating by parts and using the boundary conditions.
We first estimate the  terms $T_1, \ldots , T_6$ on the right side of \eqref{eq:diff:H1:1} as
\begin{align}
T_1 &= \la \partial_i v^{(1)} \cdot \grad v , \partial_i v \ra
 \leq \| \nabla v^{(1)}\|_{L^3} \|\nabla v\|_{L^3}^2  \leq C (M \bar M)^{1/2} R \bar R \leq \frac{1}{12} \bar R^2 + C M \bar M R^2 
 \label{eq:diff:H1:T1}\\
T_2 &= \la \partial_i w^{(1)} \dz v , \partial_i v \ra
 \leq \| \nabla w^{(1)}\|_{L^2_x L^\infty_z} \|\nabla v\|_{L^4_x L^2_z}^2 \leq C \bar M R \bar R \leq \frac{1}{12} \bar R^2 + C \bar M^2 R^2 
  \label{eq:diff:H1:T2}\\
T_3 &= \la \partial_i v \cdot \grad v^{(2)} , \partial_i v \ra
 \leq \|\grad v^{(2)}\|_{L^3} \| \nabla v\|_{L^3}^2   \leq C (M \bar M)^{1/2} R \bar R \leq \frac{1}{12} \bar R^2 + C M \bar M R^2 
  \label{eq:diff:H1:T3}\\
T_4 &= \la \partial_i w \dz v^{(2)} , \partial_i v \ra
 \leq \| \nabla w\|_{L^2_x L^\infty_z} \|\dz v^{(2)}\|_{L^4_x L^2_z} \|\nabla v\|_{L^4_x L^2_z} \notag\\
& \leq C \bar R (M \bar M)^{1/2} (R \bar R)^{1/2} \leq \frac{\bar R^2}{12}  + C M^2 \bar M^2 R^2
 \label{eq:diff:H1:T4}\\
T_5 &= \la v \cdot \grad \partial_i v^{(2)} , \partial_i v \ra
 \leq \|v \|_{L^6} \|\nabla \dz v^{(2)}\|_{L^2} \|\nabla v\|_{L^3}  \leq C R \bar M (R \bar R)^{1/2} \leq \frac{\bar R^2}{12}  + {C \bar M^{4/3}  R^2}
  \label{eq:diff:H1:T5}\\
T_6 &= \la w \dz \partial_i v^{(2)} , \partial_i v \ra
 \leq \| w\|_{L^4_x L^\infty_z} \| \dz \nabla v^{(2)} \|_{L^2_x L^2_z} \|\nabla v\|_{L^4_x L^2_z}  \notag\\
&\leq C (R \bar R)^{1/2} \bar M (R \bar R)^{1/2} \leq \frac{\bar R^2}{12}  + C \bar M^2 R^2  \label{eq:diff:H1:T6}
\end{align}
for some sufficiently large positive constant $C$. The term $T_8$ is bounded using the Lipschitz assumption \eqref{eq:sigma:cond:Lip:quant} on $\sigma$ as 
\begin{align} 
T_8 \leq C \| \nabla v^{(1)} - \nabla v^{(2)}\|_{L^2}^2 \| \sigma (v^{(1)})-\sigma(v^{(2)}) \|_{H^1}^2 \leq C R^4 \leq C M^2 R^2. \label{eq:diff:H1:T8}
\end{align}

We now integrate \eqref{eq:diff:H1:1} from $0$ to $s$, take a supremum over $s \in [0, t\wedge \tau_\kappa]$, where $\tau_\kappa = \tau_\kappa(v_0^{(1)}, v_0^{(2)})$ is the stopping time defined in \eqref{eq:tau:kappa:2} above, appeal to \eqref{eq:diff:H1:T1}--\eqref{eq:diff:H1:T8}, and take expected values, to obtain
\begin{align} 
\frac 12 \EE \sup_{[0,t \wedge \tau_\kappa]} R^2 + \EE \int_0^{t\wedge \tau_\kappa} \bar R^2 ds 
& \leq \frac 12 \EE R(0)^2 + C \EE \int_0^{t \wedge \tau_\kappa} (1  + M^2 + \bar M^2 + M \bar M  + M^2 \bar M^2) R^2 ds \notag\\
&\qquad + \EE \sup_{s \in [0,t \wedge \tau_\kappa]} \left| \sum_k \int_0^s \la \partial_i \sigma_k(v^{(1)}) - \partial_i \sigma_k(v^{(2)}), \partial_i v \ra dW^k \right|.
\label{eq:diff:H1:2}
\end{align}
Using the Burkholder-Davis-Gundy inequality and the condition \eqref{eq:sigma:cond:Lip:quant} on the noise, we thus obtain 
\begin{align} 
\frac 12 \EE \sup_{[0,t \wedge \tau_\kappa]} R^2 + \EE \int_0^{t\wedge \tau_\kappa} \bar R^2 ds 
& \leq \frac 12 \EE R(0)^2 + C \EE \int_0^{t \wedge \tau_\kappa} ( 1 + M^2 + \bar M^2 +   M^2 \bar M^2) R^2 ds \notag\\
&\qquad + C \EE \left( \int_0^{t\wedge \tau_\kappa} R^4 ds \right)^{1/2} \notag\\
& \leq \frac 12 \EE R(0)^2 + C \EE \int_0^{t \wedge \tau_\kappa} ( 1 + M^2 + \bar M^2 +   M^2 \bar M^2) R^2 ds \notag\\
&\qquad + \frac 14 \EE \sup_{[0,t \wedge \tau_\kappa]} R^2 +  C \EE  \int_0^{t\wedge \tau_\kappa} R^2 ds.  
\label{eq:diff:H1:3}
\end{align}
We now apply the stochastic Gr\"onwall Lemma (see~\cite{GZ1,DGHTZ12}) which combined with the definition of $\tau_\kappa$ in \eqref{eq:tau:kappa}--\eqref{eq:tau:kappa:2} implies
\begin{align*} 
\EE \sup_{[0, t\wedge \tau_\kappa]} \| v^{(1)}-v^{(2)}\|_{H^1}^2 \leq \EE \sup_{[0,t\wedge \tau_\kappa]} R^2  \leq C e^{C (\kappa+t)}  \EE \|v_0^{(1)}-v_0^{(2)}\|_{H^1}^2 
\end{align*}
concluding the proof of the lemma.
\end{proof}

\section{Existence of an invariant measure} \label{sec:existence}
The classical technique for proving the existence of invariant measures is the Kryloff-Bogoliubov procedure. The following well-known result (see e.g.~\cite{ZabczykDaPrato1996,Debussche2011a,KuksinShirikian12}) applies to any Feller Markov semigroup.

\begin{lemma}[\bf K-B procedure] 
\label{lem:sufficient}
Assume $P_{t}$ is Feller on $V$, and there exists $v_{0} \in L^{2}(\Omega,V)$ such that the family of Borel probability measures on $V$
\begin{align*}
\mu_T(\cdot) = \frac{1}{T} \int_0^T P_{t}(v_{0} , \cdot ) dt 
\end{align*}
is tight. Then any sub-sequential weak limit $\mu$ of the family $\{ \mu_T\}_{T>0}$ is an invariant measure for $P_t$.
\end{lemma} 
In view of the results proven in Section~\ref{sec:Feller}, the Markov semigroup $P_t$ of the stochastic Primitive Equations defined in \eqref{eq:Pt:def} is Feller. Hence, Lemma~\ref{lem:sufficient} yields the proof of Theorem~\ref{thm:existence} once the tightness of the family $\{\mu_T\}_{T>0}$ is established. To prove tightness, we obtain in the next teorem a  moment bound for strong solutions $v$ of \eqref{eq:P:1}--\eqref{eq:BC:3}, in a space that compactly embeds in $V$. 
\begin{theorem}[\bf Strong moments]
\label{thm:moments}
Let $v_0 \in L^2(\Omega;\DD(A))$ and let $v(t,v_0)$ be the strong solution of the stochastic 3D Primitive Equations \eqref{eq:P:1}--\eqref{eq:BC:3}.
Then, for any {deterministic time} $T>0$ we have
\begin{align}
& \EE \int_0^T \log(1 + \| \nabla ( |v(s)|^7) \|_{L^2}^2 + \| \nabla (| \dz v(s)|^3) \|_{L^2}^2 + \| \Delta v(s)\|_{L^2}^2 ) ds \notag\\
&\qquad \leq C\EE \log(1+ \|v_0\|_{L^{14}}^{14} + \| \dz v_0 \|_{L^6}^6 + \| \nabla v_0 \|_{L^2}^2) + C \EE \|v_0\|_{L^2}^2 + C T \label{eq:MAIN}
\end{align}
for some $T$-independent constant $C>0$. 
\end{theorem}
The above theorem, combined with the comments after Lemma~\ref{lem:sufficient} above, completes the proof of the existence of an invariant measure.
\begin{proof}[Proof of Theorem~\ref{thm:existence}]
To see why Theorem~\ref{thm:moments} implies the tightness in $\Pr(V)$ of the family of time-average measures $\{ \mu_T\}_{T>0}$, let $R > 0$ and denote by $B_R$ the ball of radius $R$ in $\DD(A)$, which is compact in $V$. Moreover, let the initial data $v_0 = 0$. Then, for any $\delta>0$, by Chebyshev's inequality and \eqref{eq:MAIN} we have that
\begin{align*}
 \mu_T(B_R^c) 
 &= \frac 1T \int_0^T \Prb( \| v(t,0)\|_{H^2} \geq R) dt  \notag\\ 
 &\leq   \frac{1}{T \log(1+R^2) } \int_0^T \EE \log(1 + \| v(t,0)\|_{H^2}^2) dt  \leq \frac{C}{\log(1+R^2)}  \leq \delta
\end{align*}
provided $R$ is chosen large enough, {independently of $T$}. Hence, we may apply Lemma~\ref{lem:sufficient} to obtain that any sub-sequential limit of the family $\{ \mu_T \}_{T>0}$ is an invariant measure. The existence of sub-sequential limits $\mu$ is guaranteed by Prokhorov's theorem, which establishes the weak compactness of the family $\{\mu_T\}_{T>0}$.  See e.g.~\cite{ZabczykDaPrato1996} for details. 
 
By the above argument, the set of invariant measures, henceforth denoted by ${\mathcal I}$, is not empty. One may now use a general argument (cf.~\cite{ZabczykDaPrato1996}) to show the existence of an  ph{ergodic} invariant measure. Recall that an invariant measure is ergodic if and only if it is an extremal point of ${\mathcal I}$. Directly from linearity of $P_t^\ast$, the set ${\mathcal I}$ is convex, and due to the Feller property of Theorem~\ref{thm:Feller}, ${\mathcal I}$ is  also closed. By the estimate in Theorem~\ref{thm:moments} and Lemma~\ref{lem:energy:moment}, applied to any stationary solution, we obtain that the set of invariant measures is tight, and hence ${\mathcal I}$ is compact. By Krein-Millman, it now follows that ${\mathcal I}$ has an extremal point which is ergodic.
\end{proof}

The remainder of this section is devoted to proving Theorem~\ref{thm:moments}.  For simplicity of the presentation we introduce the quantities
\begin{align}
 & E = \| v\|_{L^2},  &&\bar E = \|\nabla v \|_{L^2}, \label{eq:E:def}\\
 & J = \|v\|_{L^{14}},  &&\bar J = \| \nabla (|v|^7) \|_{L^2}^{1/7} , \label{eq:J:def}\\
 & K = \| \dz v \|_{L^6},  &&\bar K = \| \nabla (|\dz v|^3)\|_{L^2}^{1/3},\label{eq:K:def} \\
 & L = \|\nabla v\|_{L^2} = \bar E,  &&\bar L = \| \Delta v\|_{L^2}.\label{eq:L:def}
\end{align}
Recall that in Lemma~\ref{lem:energy:moment} we have obtained the moment bound
\begin{align} 
\EE E(t)^2 + \EE \int_0^T \bar E(s) ^2 ds \leq \EE E(0)^2 + C T \label{eq:E:moment:2}.
\end{align}
In the following subsection we obtain bounds for $J$, $K$, and $L$ individually.  The proof of Theorem~\ref{thm:moments} is then given in Subsection~\ref{sec:existence:moments}.

\subsection{Bounds for \texorpdfstring{$J$}{J}} \label{sec:existence:moments:J}
The It\=o lemma in $L^p$, with $p=14$, applied to \eqref{eq:P:1}, combined with the cancellation property $\langle (v \cdot \grad + w \dz) v, v |v|^{12}\rangle = 0$, gives
\begin{align}
d J^{14}
&= \left( - \frac{14 \cdot 13}{49} \bar J^{14} + 14 \langle \grad p, v |v|^{12} \rangle + \frac{14 \cdot 13}{2} \sum_k \langle \sigma_k(v)^2, |v|^{12} \rangle \right) dt + 14 \sum_k \langle \sigma_k(v), v |v|^{12}\rangle dW^k_t\notag\\ 
& =: A_J dt + dM_{t,J}. \label{eq:J:Ito}
\end{align}
We next give a (pathwise) estimate for $A_J$. To bound the pressure term, use \eqref{eq:P:3} and \eqref{eq:p} to obtain
\begin{align*} 
\langle \grad p, v |v|^{12} \rangle 
&= \int \grad p \cdot v |v|^{12} dx dz  = \int \grad p \cdot M(v |v|^{12}) dx \notag\\
& \leq \|\grad p \|_{L^{7/4}_x}  \|M(v|v|^{12})\|_{L^{7/3}_x} \leq C \| M ( v\cdot \grad v + v \grad \cdot v)  \|_{L^{7/4}_x}  \|M(v|v|^{12})\|_{L^{7/3}_x}
\end{align*}
and thus
\begin{align}
\langle \grad p, v |v|^{12} \rangle 
&\leq C \left( \| v\|_{L^{14}_{x,z}} \|\nabla v\|_{L^2_{x,z}}  \right) 
	  \left( \| \grad M (v |v|^{12}) \|_{L^{14/13}_x} + \|M (v |v|^{12}) \|_{L^{14/13}_x} \right) \notag\\
&\leq C \| v\|_{L^{14} } \|\nabla v\|_{L^2} 
	  \left(\| |v|^6 \|_{L^{7/3}_{x,z}} \| \nabla (|v|^7)\|_{L^2_{x,z}} + \|v\|_{L^{14}_{x,z}}^{13} \right)  \notag\\
&\leq C J \bar E \left( J^6 \bar J^7 + J^{13}  \right)\notag\\
&\leq \frac{1}{7} \bar J^{14} + C J^{14} (\bar E + \bar E^2), \label{eq:AJ:1}
\end{align}
where we use $\grad (v |v|^{12}) = \frac{13}{7} v |v|^5 \grad ( |v|^7)$. 
The term in $A_J$ coming from the It\=o correction is estimated using \eqref{eq:sigma:cond:Bnd:L14} as
\begin{align} 
91 \sum_k \langle \sigma_k (v)^2, |v|^{12} \rangle \leq C \|v\|_{L^{14}}^{12} \| \sigma (v)\|_{L_2(L^{14})}^2 \leq C J^{12}(1+J^2)
\label{eq:sigma:est:L14}
\end{align}
which combined with \eqref{eq:AJ:1} yields
\begin{align} 
A_J \leq - \bar J^{14} + C J^{14} (1 + \bar E^2) + C \label{eq:AJ}
\end{align}
for some positive constant $C$.

\subsection{Bounds for \texorpdfstring{$K$}{K}} \label{sec:existence:moments:K} 
The It\=o lemma in $L^p$, with $p=6$, applied to the vorticity formulation \eqref{eq:P:Vort}, combined with the cancellation property $\langle (v \cdot \grad + w \dz) \dz v, \dz v |\dz v|^{4}\rangle = 0$, gives
\begin{align} 
d K^6 
&= \left( - \frac{6 \cdot 5}{9} \bar K^6 + 6 \langle - \dz v \cdot \grad v + (\grad \cdot v) \dz v, \dz v |\dz v|^4 \rangle   + \frac{6 \cdot 5}{2} \sum_k \langle (\dz \sigma_k(v))^2, |\dz v|^4\rangle \right) dt \notag\\
& \qquad \qquad + 6 \sum_k \langle \dz \sigma_k(v), \dz v |\dz v|^4 \rangle dW^k_t\notag\\
&=: A_K dt + d M_{t,K} \label{eq:K:Ito}.
\end{align}
To obtain a (pathwise) bound for $A_K$, note that upon integrating by parts in $x$ we have
\begin{align} 
\langle - \dz v \cdot \grad v  &+  (\grad \cdot v) \dz v , \dz v |\dz v|^4 \rangle 
= - \int (\dz v \cdot \grad v) \cdot \dz v |\dz v|^4 +  \int (\grad \cdot v) |\dz v|^6 \notag\\
&\leq C \int |v| | \nabla (|\dz v|^3)|  |\dz v|^3 
\leq C \| v\|_{L^{14}} \| \nabla (|\dz v|^3) \|_{L^2} \| |\dz v|^3 \|_{L^{7/3}} \notag\\
& \leq C J \bar K^3 \| \dz v \|_{L^7}^3
\leq C J \bar K^3 ( J^{1/4} \bar K^{3/4} )^3 \notag\\
&\leq \frac 12 \bar K^6 + C J^{14}, \label{eq:AK:1}
\end{align}
where in the second to last inequality above we  appealed to the estimate \eqref{eq:vz:interpolation} in Lemma~\ref{lem:vz:interpolation} below with $p=7$. After using condition \eqref{eq:sigma:cond:Bnd:W1z6} to bound
\begin{align} 
15 \sum_k \langle (\dz \sigma_k(v))^2, |\dz v|^4\rangle \leq C \| \dz v\|_{L^6}^4 \| \dz \sigma(v) \|_{L_2(L^6)}^2 \leq C K^4 ( 1+ K^2),
\label{eq:sigma:est:W1z6}
\end{align}
we obtain from \eqref{eq:AK:1} that
\begin{align}
A_K \leq - \bar K^6 + C ( J^{14} + K^6 + 1) \label{eq:AK}
\end{align}
for a suitable positive constant $C$.

\begin{lemma}[\bf Vorticity interpolation]
\label{lem:vz:interpolation}
Assume $v \in \DD(A)$. We have
\begin{align} 
\| \dz v \|_{L^{p}} \leq C \| v\|_{L^{2p/(8-p)}}^{1/4} \| \dz ( |\dz v|^3 ) \|_{L^2}^{1/4} 
\label{eq:vz:interpolation}
\end{align}
for all $p \in [4,8]$, and some constant $C$ that is bounded for $p\in[4,8]$.
\end{lemma}
\begin{proof}[Proof of Lemma~\ref{lem:vz:interpolation}]
 Let $p\in [4,8)$. Using $\dz |\dz v| = \sgn(\dz v) \partial_{zz} v$, we have 
\begin{align*}
 \int |\dz v|^p dx dz &= \int \dz v \dz v |\dz v|^{p-2} dx dz
= - (p-1) \int v |\dz v|^{p-2} \partial_{zz} v dx dz \\
 &= - \frac{p-1}{3} \int v |\dz v|^{p-4} \sgn(\dz v) \partial_z \left( |\dz v|^3 \right) dx dz\\
 &\leq \frac{p-1}{3} \| v \|_{L^{(2p)/(8-p)}} \| \dz v\|_{L^{p}}^{p-4} \| \dz ( |\dz v|^3 ) \|_{L^2}
\end{align*}
and therefore \eqref{eq:vz:interpolation} holds with $C = ((p-1)/3)^{1/4}$. In order to include the limiting case $p=8$, pass $p\to 8$ in the above estimate.
\end{proof}

\subsection{Bounds for \texorpdfstring{$L$}{L}} \label{sec:existence:moments:L} Apply $\nabla$ to \eqref{eq:P:1}, and use the $L^2$ It\=o lemma on the resulting equation  to arrive at
\begin{align}
d L^2 
&= \left( - 2 \bar L^2 + 2 \langle  v \cdot \grad v, \Delta v \rangle + 2 \langle  w \dz v, \Delta v\rangle  + \sum_k   \|\nabla \sigma_k(v)\|_{L^2}^2 \right) dt + 2 \sum_k \langle \nabla \sigma_k(v), \nabla v\rangle dW^k_t \notag\\
&=: A_L dt +d M_{t,L}.
\label{eq:L:Ito}
\end{align} 
In order to obtain a bound for $A_L$, we need to estimate the nonlinear terms $I_1 =  \langle  v \cdot \grad v, \Delta v \rangle$ and $I_2 =  \langle  w \dz v, \Delta v \rangle$. For the first term we have
\begin{align}
| I_1| 
&\leq \|v\|_{L^{14}} \|\grad v\|_{L^{7/3}} \|\Delta v\|_{L^2} \leq C \|v\|_{L^{14}} \left( \| \nabla v\|_{L^2}^{11/14} \|\Delta v\|_{L^2}^{3/14} \right)  \|\Delta v\|_{L^2}  \notag\\
&\leq C J (\bar E^{11/14} \bar L^{3/14}) \bar L \leq \frac{1}{2} \bar L^2 + C \bar E^2 J^{28/11} \label{eq:AL:1}
\end{align}
where we used that for $f = \grad v$, which has zero mean on $\TT^3$, the bound
\begin{align*}
\| f\|_{L^{7/3}} \leq C \| f\|_{L^{2}}^{11/14} \|\nabla f\|_{L^2}^{3/14}
\end{align*}
holds for a suitable positive constant $C$. The second nonlinear term is bounded as
\begin{align}
|I_2| 
&\leq \| w\|_{L^3} \| \dz v \|_{L^6} \|\Delta v\|_{L^2} 
\leq C \|\grad v\|_{L^3} \|\dz v\|_{L^6} \|\Delta v\|_{L^2}  \notag\\
&\leq C \left( \|\nabla v\|_{L^2}^{1/2} \|\Delta v\|_{L^2}^{1/2} \right) \|\dz v\|_{L^6} \|\Delta v\|_{L^2} \notag\\
&\leq C \bar E^{1/2} \bar L^{1/2} K \bar L \leq \frac 12 \bar L^2 + C \bar E^2 K^4. \label{eq:AL:2}
\end{align}
Also, by our assumption \eqref{eq:sigma:cond:Lip:quant} on $\sigma(v)$ we have
\begin{align*}
\sum_k   \|\nabla \sigma_k(v)\|_{L^2}^2 = \| \nabla \sigma(v)\|_{L_2(\UU,L^2)}^2 \leq C ( 1 + \|\nabla v\|_{L^2}^2 ) \leq C (1 + \bar E^2)
\end{align*}
which combined with \eqref{eq:AL:1} and \eqref{eq:AL:2} yields
\begin{align}
A_L \leq - \bar L^2 + C \bar E^2 \left( 1+J^{28/11} + K^4 \right) + C \label{eq:AL}
\end{align}
for some positive constant $C$.

\subsection{Coupled Moment Bounds for \texorpdfstring{$J,K$, and $L$}{J,K, and L}} \label{sec:existence:moments}
From \eqref{eq:J:Ito}, \eqref{eq:K:Ito}, and \eqref{eq:L:Ito} we have that the quantity
\begin{align}
X: = J^{14} + K^6 + L^2 \label{eq:X:def}
\end{align}
obeys 
\begin{align}
dX &= \left( A_J + A_K + A_L\right) dt + \left(dM_{t,J} + dM_{t,K}+dM_{t,L}\right)\notag\\
&= - \bar X dt + \left( A_J + A_K + A_L + \bar X\right) dt \notag\\
& \quad \quad + \sum_k \left( 14 \langle \sigma_k(v), v |v|^{12}\rangle + 6 \langle \dz \sigma_k(v), \dz v |\dz v|^4\rangle + 2 \langle \nabla \sigma_k(v) , \nabla v\rangle \right) dW^k_t \label{eq:X:Ito}
\end{align}
where we denoted 
\begin{align}
\bar X = \bar J^{14} + \bar K^6 + \bar L^2. \label{eq:bar:X:def}
\end{align}
Recall from~\eqref{eq:E:def}--\eqref{eq:L:def} that $E = \|v \|_{L^2}$, $J = \| v\|_{L^{14}}$, $K = \|\dz v \|_{L^6}$, and $L = \| \nabla v \|_{L^2} = \bar E$, and that the symbols with bars on top represent dissipative counterparts.

\begin{proof}[Proof of Theorem~\ref{thm:moments}]
It turns out that if we were to integrate \eqref{eq:X:Ito} and take expected values, the bounds one can obtain are only up to a certain stopping time. In order to obtain moment bounds for any deterministic time $T$, we introduce a $C^2$ increasing function $\phi \colon [0,\infty) \to [0,\infty)$  such that $\phi(r) \to \infty$ as $r \to \infty$. We apply the It\=o formula to $\phi(X)$ and obtain  
\begin{align}
&d \phi(X) + \bar X \phi'(X) \notag\\
&\quad = \phi'(X) \left( A_J + A_K + A_L + \bar X\right) dt + \phi'(X) \left( dM_{t,J} + dM_{t,K}+dM_{t,L}\right) \notag\\
& \quad \quad + \frac 12 \phi''(X) \sum_k  \Bigl( 14 \langle \sigma_k(v), v |v|^{12}\rangle + 6 \langle \dz \sigma_k(v), \dz v |\dz v|^4\rangle + 2 \langle \nabla \sigma_k(v) , \nabla v\rangle \Bigr)^2  dt \label{eq:phi:Ito}.
\end{align}
In view of the estimates \eqref{eq:AJ}, \eqref{eq:AK}, and \eqref{eq:AL} for $A_J$, $A_K$, and $A_L$, our assumptions on $\sigma(v)$, and $\phi' \geq 0$, we deduce from \eqref{eq:phi:Ito} that
\begin{align}
 d \phi(X) + \bar X \phi'(X) 
 &\leq C \phi'(X) \left( 1 + J^{14} (1 +\bar E)^2 + K^6 + \bar E^2 (1+ J^{28/11} + K^4) \right) dt \notag \\
 & \qquad + \phi'(X) \left(d M_{t,J} + dM_{t,K}+dM_{t,L}\right)\notag\\
 & \qquad +C |\phi''(X)| \left( J^{26} (1+ J^2) + K^{10} (1+K^2) + L^2 (1+L^2) \right) dt\notag\\
 &\leq C \phi'(X) \left( 1 + J^{14}  + K^6\right) (1 + \bar E^2)dt \notag \\
 & \qquad + \phi'(X) \left( dM_{t,J} + dM_{t,K}+dM_{t,L}\right)  +C |\phi''(X)| \left(1 + X \right)^2 dt \label{eq:phi:Ito:2}
\end{align}
for a sufficiently large constant $C$, that depends on $\sigma$.

We now make a specific choice for $\phi$, namely
\begin{align}
\phi(X) = \log(1 + X) \label{eq:phi:def}
\end{align}
which is a smooth increasing function on $[0,\infty)$ with
\begin{align}
 \phi'(X) = \frac{1}{1+X}, \label{eq:phi:derivative:1}
 \end{align}
 and
\begin{align}
|\phi''(X)| \leq \frac{1}{(1+X)^2}. \label{eq:phi:derivative:2}
\end{align}
For an arbitrary deterministic time  $T>0$, we integrate \eqref{eq:phi:Ito:2} from $0$ to $T$, take expected values, use  \eqref{eq:phi:def}, \eqref{eq:phi:derivative:1}, \eqref{eq:phi:derivative:2}, and the fact that $\EE ( M_{t,J} + M_{t,K}+M_{t,L} ) = 0$, to obtain
\begin{align}
\EE \log(1 +X(T)) + \EE \int_0^T \frac{\bar X(s)}{1+ X(s)} ds \leq \EE \log(1+ X(0)) + C T + C  \EE \int_0^T  \bar E(s)^2 ds
\label{eq:log:moment:1}
\end{align}
for a suitable positive constant $C$. To close the moment estimate we now simply combine \eqref{eq:log:moment:1} with the moment estimate \eqref{eq:E:moment:2} obtained from the energy inequality and conclude
\begin{align}
\EE \log(1 +X(T)) + \EE \int_0^T \frac{\bar X(s)}{1+ X(s)} ds \leq \EE \log(1+ X(0)) + C \EE E(0)^2+ C T
\label{eq:log:moment:2}
\end{align}
where $T>0$ is an arbitrary deterministic time.

The last key step is to deduce from \eqref{eq:log:moment:2}  in fact  a bound on $\EE \int_0^T \log (1 +\bar X(s)) ds$.  To achieve this, we estimate
\begin{align}
\EE \int_0^T \log(1 + \bar X(s)) ds  
&= \EE \int_0^T  \log \left(\frac{1 + \bar X(s)}{1+ X(s)} \right) ds  + \EE \int_0^T \log(1+X(s)) ds \notag\\
&\leq  \EE \int_0^T \frac{\bar X(s)}{1+X(s)} ds + \EE \int_0^T \log(1+X(s)) ds \notag\\
&\leq \EE \log(1+ X(0)) + C \EE E(0)^2 + C T + \EE \int_0^T \log(1+X(s)) ds 
\label{eq:log:moment:3}
\end{align}
where we used $\log(r) \leq r-1$ for all $r>0$. It remains to estimate the last term on the right side of \eqref{eq:log:moment:3}. If we merely use the bound from \eqref{eq:log:moment:2} on $\EE \log (1+X)$, we obtain a quadratic growth in $T$, which is too large for our purposes (we can afford at most linear growth in $T$). To overcome this difficulty, first note that 
\begin{align} 
J^{14} &= \| v\|_{L^{14}}^{14} 
\leq C \left( \| v\|_{L^6}^{1/3} \|v\|_{L^{42}}^{2/3} \right)^{14} 
= C \|v\|_{L^6}^{14/3} \| |v|^7 \|_{L^6}^{4/3} \notag\\
&\leq C \| \nabla v\|_{L^2}^{14/3} \left( \|\nabla (|v|^7) \|_{L^2} + \| |v|^7 \|_{L^2} \right)^{4/3} \notag\\
&\leq C \bar E^{14/3} \bar J^{28/3} + C \bar E^{14/3} J^{28/3} \notag\\
&\leq C \bar E^{14/3} \bar J^{28/3} + \frac 12 J^{14} + C \bar E^{14} 
\label{eq:J:interpolation}
\end{align}
and similarly
\begin{align} 
K^6 &= \| \dz v \|_{L^6}^{6} 
\leq C \left( \| \dz v \|_{L^2}^{1/4} \| \dz v \|_{L^{18}}^{3/4}\right)^6
= C \| \dz v \|_{L^2}^{3/2} \| |\dz v|^3 \|_{L^6}^{3/2} \notag\\
&\leq C \| \dz v\|_{L^2}^{3/2} \left( \| \nabla (|\dz v|^3) \|_{L^2} + \| |\dz v|^3 \|_{L^2} \right)^{3/2} \notag\\
&\leq C \bar E^{3/2} \bar K^{9/2} + C \bar E^{3/2} K^{9/2} \notag\\
&\leq C \bar E^{3/2} \bar K^{9/2} + \frac 12 K^{6}+ C \bar E^{6}
\label{eq:K:interpolation}
\end{align}
for a positive constant $C$. Recalling that $L^2 = \bar E^2$, we thus obtain
\begin{align*} 
1 + X 
& = 1 + \bar J^{14} + \bar K^6 + \bar E^2 \notag\\
&\leq C \left( 1+ \bar E^2 \right)^{7} + C \bar E^{14/3} \bar J^{28/3} + C \bar E^{3/2} \bar K^{9/2} \notag\\
& \leq C \left( 1+ \bar E^2 \right)^{7} + C \bar E^{14/3}  \bar X^{2/3} + C \bar E^{3/2} \bar X^{3/4} 
\leq C (1+ \bar E^2)^7 (1+ \bar X)^{3/4}
\end{align*}
and hence
\begin{align} 
\log (1+ X) \leq C + 7 \log (1+ \bar E^2) + \frac 34 \log (1+\bar X) \leq C (1 + \bar E^2) + \frac 34 \log(1+ \bar X) \label{eq:X:absorb}.
\end{align}
Combining \eqref{eq:log:moment:3} with \eqref{eq:X:absorb} yields the desired estimate
\begin{align} 
\frac 14 \EE \int_0^T \log(1 + \bar X(s)) ds
&\leq \EE \log(1+ X(0)) + C \EE E(0)^2 + C T + C  \EE \int_0^T (1 + \bar E(s)^2) ds \notag\\
&\leq\EE \log(1+ X(0)) + C \EE E(0)^2 + C T \label{eq:moment:existence}
\end{align}
where in the last inequality we have appealed to \eqref{eq:E:moment:2}. This concludes the proof of the theorem.
\end{proof}

\section{Regularity of invariant measures} \label{sec:regularity}

The purpose of this section is to give the proof of Theorem~\ref{thm:regularity}. As first step, we claim that
\begin{align} 
\int_{V} \log\Bigl(1 + \| \nabla (|v|^{3}) \|_{L^{2}}^{2} + \| \nabla \dz v\|_{L^2}^2 \Bigr) d \mu (v) < \infty.
\label{eq:first:higher:moment}
\end{align}
Lemma~\ref{lem:Y:moment} shows that for any $v_0 \in L^2(\Omega,V)$ and any deterministic time $T>0$, we have
\begin{align*} 
&\EE \sup_{t \in [0,T]} \log(1 + \| v(t)\|_{L^6}^6 + \| \dz v\|_{L^2}^2 ) + \EE \int_0^T  \frac{\| \nabla(|v(t)|^3) \|_{L^2}^2 + \|\nabla \dz v(t)\|_{L^2}^2}{1 + \| v(t)\|_{L^6}^6 + \| \dz v\|_{L^2}^2} dt \notag\\
&\qquad  \leq C \EE \log(1 + \| \nabla v_0\|_{L^2}^2) + C \EE \|v_0\|_{L^2}^2 + CT.
\end{align*}
Therefore, similarly to \eqref{eq:log:moment:3} and appealing to \eqref{lem:energy:moment}, it follows that 
\begin{align} 
\EE \int_0^T \log(1 + \| \nabla(|v(t)|^3) \|_{L^2}^2 + \|\nabla \dz v(t)\|_{L^2}^2) \leq  C \EE \|v_0\|_{H^1}^2 + C(1+T).
\label{eq:higher:moment:1}
\end{align}
Now for any $n \geq 1$, $R>0$, and $v \in V$, let 
\begin{align*} 
f_{n,R} (v) = \left( \log( 1 + \|\nabla P_n   (| v|^3)\|_{L^2}^2 + \|\nabla P_n  \dz  v \|_{L^2}^2 \right) \wedge R
\end{align*}
where $P_n$ is the projection operator onto the span of the first $n$ eigenfunctions of $-\Delta$. 
Hence, by the definition of the measure $\mu$ being invariant we have  
\begin{align} 
\int_V f_{n,R}(v_0) d\mu(v_0) = \int_V \int_V \frac{1}{T} \int_0^T P_t (v_0, dv) f_{n,R}(v) dt d\mu (v_0)
\label{eq:higher:moment:2}
\end{align}
for any $T>0$. Now, for any $\rho\geq 1$ and any $v_0$ in $B_V(\rho)$, the ball or radius $\rho$ about the origin in $V$, we have by \eqref{eq:higher:moment:1}
that
\begin{align} 
& \left| \frac{1}{T} \int_0^T \int_V P_t (v_0, dv) f_{n,R}(v) dt \right| \notag\\
&\qquad = \left| \frac{1}{T} \int_0^T \EE f_{n,R}(v(t,v_0)) dt \right| \leq C \left( 1 + \frac{1 +  \|v_0\|_{H^1}^2}{T} \right) \leq C \left( 1 + \frac{\rho}{T} \right).
\label{eq:higher:moment:3}
\end{align}
Therefore we may combine \eqref{eq:higher:moment:2} and \eqref{eq:higher:moment:3} and obtain
\begin{align} 
\int_V f_{n,R}(v_0) d\mu(v_0) 
&\leq \int_{B_V(\rho)} \left| \frac{1}{T} \int_0^T \int_V P_t (v_0, dv) f_{n,R}(v) dv dt \right|   d\mu(v_0)  \notag\\
&\qquad \qquad+ \int_{V \setminus B_V(\rho)} \left| \frac{1}{T} \int_0^T \int_V P_t (v_0, dv) f_{n,R}(v) dv dt \right|   d\mu(v_0) \notag\\
&\leq C (1+ \rho T^{-1}) \mu(B_V(\rho)) + R \mu(V \setminus B_V(\rho))
\label{eq:higher:moment:4}
\end{align}
for some positive constant $C$, that does not depend on $T,\rho,R$.
First let $\rho$ be sufficiently large, depending on $R$, so that
\begin{align*} 
R \mu( V \setminus B_V(\rho) ) \leq 1.
\end{align*}
Then we choose $T$ be sufficiently large, depending on $\rho$, so that
\begin{align*} 
\rho T^{-1} \leq 1
\end{align*}
and obtain from \eqref{eq:higher:moment:4} that
\begin{align} 
\int_V f_{n,R}(v_0) d\mu(v_0)  \leq C
\label{eq:higher:moment:5}
\end{align}
independently of $n$ and $R$. We apply the Fatou lemma to \eqref{eq:higher:moment:5} as $n \to \infty$ and obtain
\begin{align} 
\int_V \left( \log(1 + \| \nabla(|v_0|^3) \|_{L^2}^2 + \|\nabla \dz v_0\|_{L^2}^2) \wedge R \right) d\mu(v_0) \leq C.
\label{eq:higher:moment:6}
\end{align}
To obtain \eqref{eq:first:higher:moment}, we apply the Monotone Convergence Theorem to \eqref{eq:higher:moment:6}.
The above bound in particular implies that the invariant measure $\mu$ is supported on the space
\begin{align}
{\mathcal Y} = \left\{ v \in V \colon \nabla (|v|^3) \in L^2, \nabla \dz v \in L^2 \right\}.\label{eq:def:Yspace}
\end{align}
Note that in particular, for $v\in {\mathcal Y}$ we have $v \in L^{14}$ and $\dz v \in L^6$. 
Having established \eqref{eq:first:higher:moment}, we proceed to prove the higher regularity for the support of the invariant measures.

\begin{proof}[Proof of Theorem~\ref{thm:regularity}]
We need to show that
\begin{align} 
\int_{V} \log\Bigl(1 + \| \nabla (|v|^{7}) \|_{L^{2}}^{2} + \| \nabla (|\dz v|^3)\|_{L^2}^2  + \| \Delta v\|_{L^2}^2 \Bigr) d \mu (v) < \infty
\label{eq:second:higher:moment}
\end{align}
which in particular implies that the invariant measure is supported on the space
\begin{align}
{\mathcal X} = \left\{ v \in H^2 \cap H \colon \nabla (|v|^7) \in L^2, \nabla (|\dz v|^3) \in L^2 \right\} \label{eq:def:X:space}
\end{align}
and hence in particular on $H^2$. The proof of \eqref{eq:second:higher:moment} follows the same argument as the proof of \eqref{eq:first:higher:moment}, but instead of appealing to the moment bound in Lemma~\ref{lem:Y:moment}, we appeal to the strong moment bound given in theorem~\ref{thm:moments}.

To obtain \eqref{eq:second:higher:moment}, for $v \in {\mathcal Y}$ we define
\begin{align*} 
g_{n,R} (v) = \log( 1 + \|\nabla P_n   (| v|^7)\|_{L^2}^2 + \|\nabla P_n  (|\dz  v|^3) \|_{L^2}^2 + \| \Delta P_n v\|_{L^2}^2) \wedge R.
\end{align*}
Using the moment bound in Theorem~\ref{thm:moments} which for $v_0 \in {\mathcal Y}$ yields
\begin{align}
\EE \int_0^T g_{n,R}(v(s)) ds &\leq \EE \int_0^T \log(1 + \| \nabla ( |v(s)|^7) \|_{L^2}^2 + \| \nabla (| \dz v(s)|^3) \|_{L^2}^2 + \| \Delta v(s)\|_{L^2}^2 ) ds \notag\\
&\qquad \leq C\EE \log(1+ \|v_0\|_{L^{14}}^{14} + \| \dz v_0 \|_{L^6}^6 + \| \nabla v_0 \|_{L^2}^2) + C \EE \|v_0\|_{L^2}^2 + C T \notag\\
&\qquad \leq C \log(1+ \|\nabla (|v_0|^3)\|_{L^{2}}^{2} + \| \nabla \dz v_0 \|_{L^2}^2)  + C  \|v_0\|_{H^1}^2 + C T \label{eq:MAIN:2}
\end{align}
where in the last inequality we used the Sobolev embedding $H^1 \subset L^6$. Then, the equivalent of \eqref{eq:higher:moment:2} holds with $g_{n,R}$ replacing $f_{n,R}$, since $\mu$ is invariant. Repeating the argument in \eqref{eq:higher:moment:4}, with $B_V(\rho)$ 
replaced by $B_{\mathcal Y}(\rho)$, the ball of radius $\rho$ in ${\mathcal Y}$, we obtain by using that $\mu$ is supported on ${\mathcal Y}$, choosing $\rho$ sufficiently large, and then $T$ sufficiently large, that 
\begin{align*}
\int_V g_{n,R}(v_0) d\mu(v_0) = \int_{\mathcal Y} g_{n,R}(v_0) d\mu(v_0) \leq C
\end{align*}
for a constant $C$ independent of $n$ and $R$. With the Fatou lemma and the monotone convergence theorem we then conclude the proof of \eqref{eq:second:higher:moment}.
\end{proof}

\section{Higher regularity of invariant measures} \label{sec:higher:regularity}
In Section~\ref{sec:existence} we have proven the existence of an invariant measure for the 3D stochastic PEs, which is supported on $H^2(\TT^3) \cap H$ by the results in Section~\ref{sec:regularity}. The purpose of this section is to prove Theorem~\ref{thm:higher:regularity}, i.e., that the invariant measures are in fact supported on smoother functions, more precisely, on $H^{2+\eps}(\TT^3) \cap H$, where $\eps \in [0,1/42]$. {Under suitable conditions on $\sigma(v)$ we in fact conjecture that the invariant measures are supported on $C^\infty(\TT^3) \cap H$.} 

For this purpose, fix $\eps \in [0,1/42]$ for the rest of this section, and besides the quantities $E, \bar E$ defined in \eqref{eq:E:def}, $J, \bar J$ defined in \eqref{eq:J:def}, and $K, \bar K$ defined in \eqref{eq:K:def}, we introduce
\begin{align}
L_\eps = \| D^\eps \nabla v\|_{L^2}, \quad \bar L_\eps = \| D^{1+\eps} \nabla v\|_{L^2} \label{eq:L:eps:def},
\end{align}
where $D = \sqrt{-\Delta}$. 
The main ingredient in the proof of higher regularity is the following theorem.
\begin{theorem}[\bf Higher Moments] \label{thm:higher:moments:eps}
Let $v_0 \in L^2(\Omega;\DD(A))$ and let $v(t) = v(t,v_0)$ be the strong solution of the stochastic 3D Primitive Equations \eqref{eq:P:1}--\eqref{eq:BC:3}, and let $\eps \in [0,1/42]$.
Then, for any {deterministic time} $T>0$ we have 
\begin{align}
&\EE \int_0^T \log \Bigl( 1 + \| \nabla (|v(t)|^7)\|_{L^2}^2 + \|\nabla ( |\dz v(t)|^3)\|_{L^2}^2 + \| D^{1+\eps} \nabla v(t) \|_{L^2}^2 \Bigr) dt \notag\\
& \qquad \leq C \EE \log\left(1 + \|v_0\|_{L^{14}}^{14} + \|\dz v_0\|_{L^6}^6 + \| D^{1+\eps} v_0\|_{L^2}^2 \right) + C \EE \|v_0\|_{L^2}^2 + C (1 + T)
\label{eq:higher:moments:eps}
\end{align}
for a suitable positive constant $C$.
\end{theorem}

Assuming the moment bound \eqref{eq:higher:moments:eps} holds, let us first give the proof of higher regularity for the invariant measures. 
\begin{proof}[Proof of Theorem~\ref{thm:higher:regularity}]
 We need to show that for an invariant measure $\mu$ we have
\begin{align*} 
\int_{V} \Bigl( 1 + \|\nabla (|v|^7) \|_{L^2}^2 + \|\nabla (|\dz v|^3)\|_{L^2}^2 + \| D^{1+\eps} \nabla v\|_{L^2}^2 \Bigr) d\mu(v) < \infty
\end{align*}
We already know from Theorem~\ref{thm:regularity} that invariant measures are supported on the space ${\mathcal X}$ defined in \eqref{eq:def:X:space}, and \eqref{eq:second:higher:moment} holds. As in the proof of Theorem~\ref{thm:regularity}, we define for $v \in {\mathcal X}$ the function
\begin{align*} 
g_{n,R} (v) = \log( 1 + \|\nabla P_n   (| v|^7)\|_{L^2}^2 + \|\nabla P_n  (|\dz  v|^3) \|_{L^2}^2 + \| \nabla P_n D^{1+\eps} v\|_{L^2}^2) \wedge R
\end{align*}
with $n \geq 1$, and $R>0$. Using the moment bound \eqref{eq:higher:moments:eps} we may show that for $v_0 \in {\mathcal X}$ we have
\begin{align*} 
\EE \int_0^T g_{n,R}(v(s))ds 
&\leq C \EE \log\left(1 + \|v_0\|_{L^{14}}^{14} + \|\dz v_0\|_{L^6}^6 + \| D^{1+\eps} v_0\|_{L^2}^2 \right) + C \EE \|v_0\|_{L^2}^2 + C (1 + T)\notag\\
&\leq C \EE \|v_0\|_{H^2}^2 + C(1+T)
\end{align*}
with $C$ independent of $R$ and $n$. 
The same arguments as in Section~\ref{sec:regularity}, estimates \eqref{eq:higher:moment:2}--\eqref{eq:higher:moment:5}, give that 
\begin{align*} 
\int_{V} g_{n,R} (v_0) d\mu(v_0) = \int_{\mathcal X} g_{n,R} (v_0) d\mu(v_0) \leq C
\end{align*}
holds uniformly in $n$ and $R$. We conclude the proof with the Fatou lemma and the monotone convergence theorem. In particular, this shows that the invariant measure $\mu$ is supported on $H^{2+\eps}$.
\end{proof}

The rest of this section is devoted to the proof of Theorem~\ref{thm:higher:moments:eps}.
We recall from \eqref{eq:J:Ito}, \eqref{eq:AJ}, \eqref{eq:K:Ito}, and \eqref{eq:AK} that for $J$ and $K$ we have the It\=o formula
\begin{align}
d ( J^{14} + K^6 ) = (A_{J} + A_{K} ) dt + ( d M_{t,J} + d M_{t,K}), \label{eq:J+K:Ito}
\end{align}
and the bound
\begin{align}
A_J + A_K \leq  - (\bar J^{14} + \bar K^6) + C (1+ \bar E^2) (J^{14}+ K^6 + 1) \label{eq:regularity:JK:bar},
\end{align}
holds for a suitable positive constant $C$. As in the previous section for $\eps=0$, we couple \eqref{eq:J+K:Ito} with the evolution of $L_\eps$, and then obtain suitable moment bounds for $X_\eps = J^{14} + K^6 + L_\eps^2$.

\subsection{Bounds for \texorpdfstring{$L_\eps$}{L epsilon}} \label{sec:L:eps:moments}
We apply $\nabla D^\eps$ to \eqref{eq:P:1}, use It\=o Lemma pointwise in $x$, and then take an $L^2$ inner product with $D^\eps \nabla v$ to obtain
\begin{align}
d L_\eps^2 
& = \left( - 2 \bar L_\eps^2 + 2 \langle  v \cdot \grad v, \Delta D^{2\eps} v \rangle + 2 \langle  w \dz v, \Delta D^{2\eps} v\rangle  + \sum_k   \|\nabla D^\eps \sigma_k(v)\|_{L^2}^2 \right) dt \notag\\
& \qquad + 2 \sum_k \langle \nabla D^\eps \sigma_k(v), \nabla D^\eps v\rangle dW^k_t =: A_{L,\eps} dt +  dM_{t,L,\eps} \label{eq:L:eps:Ito}.
\end{align}
Note that in order to obtain \eqref{eq:L:eps:Ito} we used 
\begin{align*}
\langle \grad p , D^{2\eps} \Delta v \rangle = 0.
\end{align*}
We write 
\[ A_{L,\eps} = I_{1,\eps} + I_{2,\eps} + \| \nabla D^\eps \sigma(v)\|_{L_2(L^2)}^2,\]
where $I_{1,\eps}$ is the term arising from $v \cdot \grad$, and $I_{2,\eps}$ is the one arising due to $w \dz$. By our assumptions on $\sigma(v)$ we have
\begin{align}
 \| \nabla D^\eps \sigma(v)\|_{L_2(L^2)}^2 \leq C (1 + \| D^\eps \nabla v\|_{L^2}^2) = C (1+ L_\eps^2) \label{eq:A:J:eps:1}.
\end{align}

In order to bound $I_{1,\eps}$ we appeal to the fractional Sobolev product estimate~\cite{Taylor91}
\begin{align}
\| D^\eps ( fg) \|_{L^2} \leq C \| D^\eps f\|_{L^{p_1}} \| g\|_{L^{p_2}} + C \| f\|_{L^{p_3}} \|D^\eps g\|_{L^{p_4}} 
\label{eq:Heps:product}
\end{align}
where $1/2 = 1/p_1 + 1/p_2 = 1/p_3 + 1/p_4$, and $p_1,p_2,p_3,p_4 \in ( 2,\infty)$, to obtain
\begin{align}
|I_{1,\eps}| &= 2 |\langle  v \cdot \grad v, \Delta D^{2\eps} v\rangle| 
\leq 2 \| D^\eps \left( (v \cdot \grad) v \right)\|_{L^2} \| \Delta D^\eps v \|_{L^2}\notag\\
&\leq C \left( \| D^\eps v \|_{L^{14}} \|\nabla v\|_{L^{7/3}} + \| v\|_{L^{14}} \| D^\eps \nabla v \|_{L^{7/3}} \right) \bar L_\eps \notag\\
&=: C (I_{11} + I_{12}) \bar L_\eps.
\label{eq:J12:def}
\end{align}
Using the Sobolev inequality 
\begin{align*}
\| D^{5/7} u \|_{L^{14}} \leq C \| D^2 u \|_{L^2} 
\end{align*}
and the Gagliardo-Nirenberg inequality
\begin{align}
\| u \|_{L^{7/3}} \leq C \| u \|_{L^{2}}^{11/14} \| \nabla u \|_{L^2}^{3/14} \label{eq:GN:1}
\end{align}
we may estimate for all $\eps \in [0,1/7]$
\begin{align*}
I_{11} &= \| D^\eps v \|_{L^{14}} \|\nabla v\|_{L^{7/3}}
\leq C \|v \|_{L^{14}}^{1 - \frac{7 \eps}{5}} \| D^{5/7} v\|_{L^{14}}^{\frac{7\eps}{5}} \| \nabla v \|_{L^2}^{\frac{11}{14}} \| D^2 v \|_{L^2}^{\frac{3}{14}} \notag\\
&\leq C \|v \|_{L^{14}}^{1 - \frac{7 \eps}{5}} \| D^{2} v\|_{L^{2}}^{\frac{7\eps}{5}} \| \nabla v \|_{L^2}^{\frac{11}{14}} \| D^2 v \|_{L^2}^{\frac{3}{14}} \notag\\
&\leq C J^{1- \frac{7\eps}{5}} \bar E^{\frac{11}{14}} \bar L_\eps^{\frac{3}{14} + \frac{7\eps}{5}}.
\end{align*}
In the last inequality above we  used the Poincar\'e inequality $\|D^2 v\|_{L^2} \leq C L_\eps$. Therefore, by Young's inequality we arrive at
\begin{align}
C I_{11} \bar L_\eps 
&\leq C J^{1- \frac{7\eps}{5}} \bar E^{\frac{11}{14}} \bar L_\eps^{\frac{17}{14} + \frac{7\eps}{5}} \notag\\
&\leq \frac 14 \bar L_\eps^2 + C J^{\frac{140 - 196 \eps}{55-98\eps}} \bar E^{\frac{110}{55 - 98 \eps}}  \notag\\
&\leq \frac 14 \bar L_\eps^2 + C J^{\frac{140 - 196 \eps}{55-98\eps}} \bar E^2 L_\eps^{\frac{196 \eps}{55 - 98 \eps}} 
\label{eq:I11}
\end{align}
since $\bar E \leq C L_\eps$. To bound the $I_{12}$ term in \eqref{eq:J12:def} we proceed as follows. From \eqref{eq:GN:1} we obtain
\begin{align*}
I_{12} = \| v\|_{L^{14}} \| D^\eps \nabla v\|_{L^{7/3}} &\leq C \| v\|_{L^{14}} \|D^\eps \nabla v\|_{L^2}^{\frac{11}{14}} \| D^{1+\eps}  \nabla v\|_{L^2}^{\frac{3}{14}} \leq C J L_\eps^{\frac{11}{14}} \bar L_\eps^{\frac{3}{14}}
\end{align*}
and hence by interpolating $L_\eps \leq \bar E^{1/(1+\eps)} \bar L_\eps^{\eps/(1+\eps)}$ we arrive at
\begin{align}
C I_{12} \bar L_\eps 
&\leq C J L_\eps^{\frac{11}{14}} \bar L_\eps^{\frac{17}{14}} \leq C J \bar E^{\frac{11}{14(1+\eps)}} \bar L_\eps^{\frac{17 + 28 \eps}{14(1+\eps)}} \notag\\
&\leq \frac 14 \bar L_\eps^2 + C J^{\frac{28(1+\eps)}{11}} \bar E^2 
\label{eq:I12}.
\end{align}
Combining \eqref{eq:J12:def}, \eqref{eq:I11}, and \eqref{eq:I12} we thus obtain
\begin{align}
|I_{1,\eps}| 
& \leq \frac 12 \bar L_\eps^2 + \bar E^2 \left( J^{\frac{140 - 196 \eps}{55-98\eps}}  L_\eps^{\frac{196 \eps}{55 - 98 \eps}} + J^{\frac{28(1+\eps)}{11}} \right) \notag\\
& \leq \frac 12 \bar L_\eps^2 + \bar E^2 \left(  L_\eps^{2} + J^{\frac{280-392 \eps}{110 - 392 \eps}} + J^{\frac{28(1+\eps)}{11}} \right) \notag\\
& \leq \frac 12 \bar L_\eps^2 + \bar E^2 \left(  L_\eps^{2} + 1+  J^{14} \right) \label{eq:I1eps}
\end{align}
for some suitable constant $C$, and all $\eps \in [0,1/7]$.

It is left to estimate the $I_{2,\eps}$ term in $A_{L,\eps}$. Using \eqref{eq:Heps:product} we have
\begin{align}
|I_{2,\eps}| 
&= 2 | \langle  w \dz v, \Delta D^{2\eps} v\rangle | 
\leq C \left( \| D^\eps w\|_{L^3} \| \dz v\|_{L^6} +  \| w\|_{L^{7/2}} \| D^\eps  \dz v\|_{L^{14/3}} \right)\| D^\eps \Delta v\|_{L^2}\notag\\
&\leq C \left( I_{21} + I_{22} \right) \bar L_\eps \label{eq:J22:def}.
\end{align}
We bound the $I_{21}$ term in \eqref{eq:J22:def} as
\begin{align*}
I_{21} 
&\leq C \| D^\eps \nabla v \|_{L^3} \| \dz v \|_{L^6}  \leq C \| D^\eps \nabla v\|_{L^2}^{\frac 12 } \| D^{1+\eps} \nabla v \|_{L^2}^{\frac 12}\| \dz v \|_{L^6} \notag\\
&\leq C \left( \|\nabla v\|_{L^2}^{\frac{1}{1+\eps}} \|D^{1+\eps} \nabla v \|_{L^2}^{\frac{\eps}{1+\eps}} \right)^{\frac 12} \| D^{1+\eps} \nabla v \|_{L^2}^{\frac 12}\| \dz v \|_{L^6} \notag\\
&\leq C \bar E^{\frac{1}{2(1+\eps)}} \bar L_\eps^{\frac{1+2\eps}{2(1+\eps)}} K,
\end{align*}
and therefore
\begin{align}
C I_{21} \bar L_\eps 
&\leq C K \bar E^{\frac{1}{2(1+\eps)}} \bar L_\eps^{\frac{3+4\eps}{2(1+\eps)}} 
\leq \frac 14 \bar L_\eps^2 + C K^{4(1+\eps)} \bar E^2\notag\\
&\leq \frac 14 \bar L_\eps^2 + C (1 + K^6) \bar E^2.
\label{eq:I21}
\end{align}
To bound the $I_{22}$ term in \eqref{eq:J22:def}, we use the Gagliardo-Nirenberg inequality combined with the fact that we are on a bounded domain
\begin{align*}
\| D^\eps u\|_{L^{14/3}} \leq C \| u \|_{L^{14/3}}^{\frac{1}{1+ 7 \eps}} \| D^{1+\eps} u \|_{L^2}^{\frac{ 7 \eps}{1 + 7\eps}} 
\leq C \| u\|_{L^6}^{\frac{1}{1+ 7 \eps}} \| D^{1+\eps} u \|_{L^2}^{\frac{ 7 \eps}{1 + 7\eps}},
\end{align*}
with $u = \dz v$, and obtain
\begin{align*}
I_{22} &\leq C \| \nabla v\|_{L^{7/2}} \|D^\eps \dz v\|_{L^{14/3}}\notag\\
&\leq C \| v \|_{L^{14}}^{\frac 12} \| D^2 v \|_{L^2}^{\frac 12} \| \dz v\|_{L^6}^{\frac{1}{1 +7 \eps}} \| D^{1+\eps} \dz v \|_{L^2}^{\frac{7 \eps}{1+7\eps}} \notag\\
&\leq C J^{\frac 12} \bar L_\eps^{\frac 12}  K^{\frac{1}{1+7\eps}} \bar L_\eps^{\frac{7 \eps}{1+7 \eps}}
\end{align*}
where in the last inequality we have also used the Poincar\'e inequality $\|D^2 v\|_{L^2} \leq \bar L_\eps$. Therefore, for any $\eps\in[0,1/42]$ we obtain
\begin{align}
C I_{22} \bar L_\eps 
&\leq C J^{\frac 12} K^{\frac{1}{1+7\eps}} \bar L_\eps^{\frac{3+35 \eps}{2(1+7\eps)}} 
\leq \frac 14 \bar L_\eps^2 + C J^{\frac{2 (1 + 7 \eps)}{1 - 7 \eps}} K^{\frac{4}{1-7\eps}} \notag\\
&\leq \frac 14 \bar L_\eps^2 + C J^{\frac{6(1 + 7 \eps)}{1 - 21 \eps}} + C K^{6} \leq \frac 14 \bar L_\eps^2 + C( 1 + J^{14} +  K^{6})
\label{eq:I22}.
\end{align}
Finally, by combining \eqref{eq:J22:def}, \eqref{eq:I21}, and \eqref{eq:I22} we arrive at 
\begin{align}
|I_{2,\eps}| \leq \frac 12 \bar L_\eps^2 + C (1+J^{14} + K^6) (1+\bar E^2) \label{eq:I2eps}.
\end{align}
We conclude the estimate
\begin{align}
A_{L,\eps} \leq \bar L_\eps^2 + C (1+\bar E^2) ( 1 + J^{14} + K^6 + L_\eps^2), \label{eq:AL:eps}
\end{align}
which holds for $\eps \in [0,1/42]$ and a suitable positive constant $C$, by collecting \eqref{eq:A:J:eps:1}, \eqref{eq:I1eps}, and \eqref{eq:I2eps}.

\subsection{Coupled Moment Bounds for \texorpdfstring{$J$}{J},  \texorpdfstring{$K$}{K}, and  \texorpdfstring{$L_\eps$}{L epsilon}}
From \eqref{eq:J+K:Ito} and \eqref{eq:L:eps:Ito} we have that the quantity
\begin{align*}
X_\eps = J^{14} + K^6 + L_\eps^2
\end{align*}
obeys
\begin{align}
d X_\eps + \bar X_\eps dt = \left( A_J + A_K + A_{L,\eps} - \bar X_\eps \right) dt + (dM_{t,J} + dM_{t,K} + dM_{t,L,\eps})
\label{eq:X:eps:Ito}
\end{align}
where we let 
\begin{align*}
\bar X_\eps = \bar J^{14} + \bar K^6 + \bar L_\eps^2.
\end{align*}
We also know from \eqref{eq:regularity:JK:bar}  and \eqref{eq:AL:eps} that
\begin{align}
A_J + A_K + A_{L,\eps} - \bar X_\eps \leq  C (1+ \bar E^2) (1+ X_\eps). \label{eq:higher:drift:bound}
\end{align}

\begin{proof}[Proof of Theorem~\ref{thm:higher:moments:eps}]
We now follow precisely the same steps as in Section~\ref{sec:existence:moments}. To avoid redundancy we skip the parts which are mutatis mutandis. Letting $\phi(X_\eps) = \log(1+ X_\eps)$, as in \eqref{eq:log:moment:2}, we conclude from \eqref{eq:E:moment}, \eqref{eq:X:eps:Ito} and \eqref{eq:higher:drift:bound}, and the Ito formula for $\phi(X_\eps)$ that for any deterministic time $T>0$
\begin{align}
\EE \log(1+X_\eps(T)) + \EE \int_0^T \frac{\bar X_\eps(s)}{1+X_\eps(s)} ds \leq \EE \log(1+X_\eps(0)) + C \EE E(0)^2 + C T
\label{eq:log:moment:eps:1}
\end{align}
for a suitable positive constant $C$. As in \eqref{eq:log:moment:3}, the above estimate implies
\begin{align}
\EE \int_0^T \log(1+ \bar X_\eps(s)) ds \leq \EE \log(1+X_\eps(0)) + C \EE E(0)^2 + C T + \EE \int_0^T \log(1+X_\eps(s)) ds.
\label{eq:log:moment:eps:2}
\end{align}
To obtain the desired moment estimate that grows at most linearly in $T$ we recall from \eqref{eq:J:interpolation} and \eqref{eq:K:interpolation} that
\begin{align}
1+ J^{14} + K^6 \leq C (1+ \bar E^2)^7 (1+\bar X)^{3/4} \leq C (1+ \bar E^2)^7 (1+\bar X_\eps)^{3/4}
\label{eq:JK:inter}
\end{align}
since by the Poincar\'e inequality $\bar L = \| \Delta v\|_{L^2} \leq C \|D^\eps \Delta v\|_{L^2} = \bar L_\eps$. Combined with the interpolation inequality
\begin{align*}
L_\eps^2 = \| D^\eps \nabla v\|_{L^2}^2 \leq C \| \nabla v\|_{L^2}^{\frac{2}{1+\eps}} \| D^{1+\eps} \nabla v\|_{L^2}^{\frac{2 \eps}{1+\eps}}\leq C \bar E^{\frac{2}{1+\eps}} \| D^{1+\eps} \nabla v\|_{L^2}^{\frac{2 \eps}{1+\eps}}
\end{align*}
when $\eps \in [0,3/5]$, the estimate \eqref{eq:JK:inter} yields
\begin{align}
1 + X_\eps \leq C (1+\bar E^2)^7 (1 + \bar X_\eps)^{3/4}. 
\label{eq:Xeps:absorb}
\end{align}
Taking $\log$ of both sides of \eqref{eq:Xeps:absorb} and combining with \eqref{eq:log:moment:eps:2}, similarly to
\eqref{eq:moment:existence}, we obtain
\begin{align*}
\frac 14 \EE \int_0^T \log(1+ \bar X_\eps(s)) ds \leq \EE \log(1+X_\eps(0)) + C \EE E(0)^2 + C T
\end{align*}
for all $\eps \in [0,1/42]$ where $C$ is a positive constant.
\end{proof}

\section*{Acknowledgements}
NGH gratefully acknowledges the support of the Institute for Mathematics and its Applications (IMA) at the University of Minnesota. 
The work of NGH was supported in part by the NSF grant DMS-1313272, the work of IK was supported in part by the NSF grant DMS-1311943, the work of VV was supported in part by the NSF grant DMS-1211828, while the work of MZ was supported in part by the NSF grant DMS-1109562.

\end{document}